\documentclass[11pt,twoside, leqno]{article}

\usepackage{amssymb}
\usepackage{amsmath}
\usepackage{amsthm}

\allowdisplaybreaks

\def\rr{{\mathbb R}}
\def\rd{{{\rr}^d}}
\def\rn{{{\rr}^n}}

\def\zz{{\mathbb Z}}
\def\nn{{\mathbb N}}
\def\hh{{\mathbb H}}
\def\bbg{{\mathbb G}}

\def\cx{{\mathcal X}}

\def\lcz{{\lceil}}
\def\rcz{{\rceil}}

\def\cd{{\mathcal D}}

\def\cl{{\mathcal L}}

\def\cb{{\mathcal B}}

\def\fz{\infty}
\def\az{\alpha}

\def\loc{{\mathop\mathrm{\,loc\,}}}

\def\lz{\lambda}
\def\dz{\delta}

\def\ez{\epsilon}

\def\kz{\kappa}
\def\bz{\beta}

\def\gz{{\gamma}}

\def\wz{\widetilde}
\def\ol{\overline}

\def\gl{g_\lz^*}
\def\glz{g_{\lz,\,0}^*}
\def\glf{g_{\lz,\,\fz}^*}

\def\hs{\hspace{0.3cm}}

\def\ls{\lesssim}
\def\gs{\gtrsim}

\def\lt{{L^2(\cx)}}

\def\bmo{{\mathop\mathrm{BMO}}}
\def\sbmo{{\mathop\mathrm{bmo}}}
\def\sblo{{\mathop\mathrm{blo}}}

\def\blo{{\mathop\mathrm{BLO}}}

\def\bmoz{{\mathrm{BMO}_\rho(\cx)}}
\def\bloz{{\mathrm{BLO}_\rho(\cx)}}

\def\einf{{\mathop{\mathrm{\,essinf\,}}}}
\def\dsum{\displaystyle\sum}
\def\diam{{\mathop\mathrm{\,diam\,}}}

\def\dint{\displaystyle\int}

\def\gfz{\genfrac{}{}{0pt}{}}

\def\dfrac{\displaystyle\frac}
\def\dsup{\displaystyle\sup}

\def\r{\right}
\def\lf{\left}

\newtheorem{thm}{Theorem}[section]
\newtheorem{lem}{Lemma}[section]
\newtheorem{prop}{Proposition}[section]
\newtheorem{rem}{Remark}[section]
\newtheorem{cor}{Corollary}[section]
\newtheorem{defn}{Definition}[section]
\newtheorem{example}{Example}[section]

\numberwithin{equation}{section}

\begin{document}

\arraycolsep=1pt

\title{{\vspace{-5cm}\small\hfill\bf Bull. Sci. Math., to appear}\\
\vspace{4cm}\Large\bf Boundedness of Lusin-area and $g_\lz^*$ Functions on
Localized BMO Spaces over Doubling Metric Measure
Spaces\footnotetext{\hspace{-0.35cm} 2000 {\it Mathematics Subject
Classification}. {Primary 42B25; Secondary 42B30, 51F99.}
\endgraf{\it Key words and phrases.} doubling metric measure space,
chain ball property, weak (monotone) geodesic property,
annular decay property, admissible
function, Schr\"odinger operator, localized BMO space, localized BLO
space, Lusin-area function, $\gl$ function.
\endgraf
The first author was supported by Chinese Universities Scientific
Fund (Grant No. 2009JS34), the second author was supported by
Grant-in-Aid for Scientific Research (C), No. 20540167, Japan
Society for the Promotion of Science, and the third (corresponding)
author was supported by the National Natural Science Foundation
(Grant No. 10871025) of China.}}
\author{Haibo Lin, Eiichi Nakai and Dachun Yang\,\footnote{Corresponding
author.}
\vspace{0.3cm}\\
Dedicated to Professor K\^oz\^o Yabuta in celebration of his 70th
birthday
}
\date{ }
\maketitle

\vspace{-0.8cm}

\begin{center}
\begin{minipage}{13.6cm}\small
{\noindent{\bf Abstract.}
Let ${\mathcal X}$ be a doubling metric
measure space. If ${\mathcal X}$ has the $\delta$-annular decay
property for some $\delta\in (0,\,1]$, the authors then
establish the boundedness of the
Lusin-area function, which is defined via kernels modeled on the
semigroup generated by the Schr\"odinger operator, from localized
spaces ${\rm BMO}_\rho({\mathcal X})$ to ${\rm BLO}_\rho({\mathcal
X})$ without invoking any regularity of considered kernels. The same
is true for the $\gl$ function and unlike the Lusin-area
function, in this case, ${\mathcal X}$ is not necessary to have the
$\dz$-annular decay property. Moreover, for any metric space, the authors
introduce the weak geodesic property and the monotone geodesic property,
which are proved to be respectively equivalent to
the chain ball property of Buckley. Recall that Buckley proved
that any length space has the chain ball property
and, for any metric space equipped with a doubling measure,
the chain ball property implies the $\delta$-annular decay property
for some $\dz\in (0,1]$. Moreover, using some results on pointwise multipliers of
${\rm bmo}(\rr)$, the authors construct a counterexample to show
that there exists a nonnegative function which is in ${\rm
bmo}(\rr)$, but not in ${\rm blo}(\rr)$; this further indicates that
the above boundedness of the Lusin-area and $\gl$ functions even in $\rd$
with the Lebesgue measure or the Heisenberg
group also improves the existing results.}
\end{minipage}
\end{center}

\section{Introduction\label{s1}}

\hskip\parindent Since the space ${\rm BMO}(\rd)$ of functions
with bounded mean oscillation on $\rd$ was introduced by John and
Nirenberg \cite{jn61}, it then plays an important role in harmonic
analysis and partial differential equations. It is well-known that
${\rm BMO}(\rd)$ is the dual space of the Hardy space $H^1(\rd)$
(see, for example, \cite{s93,gr08}), and also a good substitute of
$L^\fz(\rd)$ in the study of boundedness of operators. However, the
space ${\rm BMO}(\rd)$ is essentially related to the Laplacian
$\Delta$, where $\Delta\equiv\sum_{j=1}^d \frac{\partial^2}{\partial
x_j^2}$.

On the other hand, there exists an increasing interest on the study
of Schr\"odinger operators on $\rd$ and the sub-Laplace
Schr\"odinger operators on connected and simply connected nilpotent
Lie groups with nonnegative potentials satisfying the reverse
H\"older inequality; see, for example,
\cite{f83,z99,s95,l99,d05,dz99,dgmtz05,ll08,yz08,hl1,hl2}. Let
$\cl\equiv -\Delta+V$ be the Schr\"odinger operator on $\rd$, where
the potential $V$ is a nonnegative locally integrable function.
Denote by $\cb_q(\rd)$ the class of nonnegative functions satisfying
the reverse H\"older inequality of order $q$. For
$V\in\cb_{d/2}(\rd)$ with $d\ge3$, Dziuba\'nski et al
\cite{d05,dz99,dgmtz05} studied the BMO-type space $\bmo_\cl(\rd)$
and the Hardy space $H^1_\cl(\rd)$ and, especially, proved that the
dual space of $H^1_\cl(\rd)$ is $\bmo_\cl(\rd)$. Moreover, they
obtained the boundedness on these spaces of the Littlewood-Paley
$g$-function associated to $\cl$. Let $\cx$ be an RD-space in
\cite{hmy2}, which means that $\cx$ is a space of homogeneous type
in the sense of Coifman and Weiss \cite{cw71, cw77} with the
additional property that a reverse doubling condition holds. Let
$\rho$ be a given admissible function modeled on the known auxiliary
function determined by $V\in\cb_{d/2}(\rd)$ (see \cite{yz08} or
\eqref{e2.3} below). The localized Hardy space $H^1_\rho(\cx)$, the
BMO-type space $\mathrm{\,BMO}_\rho({\mathcal X})$ and the BLO-type
space $\mathrm{\,BLO}_\rho({\mathcal X})$ associated with $\rho$
were introduced and studied in \cite{yz08,yyz}. Moreover, the
boundedness from $\mathrm{\,BMO}_\rho({\mathcal X})$ to
$\mathrm{\,BLO}_\rho({\mathcal X})$ of several maximal operators and
the Littlewood-Paley $g$-function, which are defined via kernels
modeled on the semigroup generated by the Schr\"odinger operator,
was obtained in \cite{yyz}.

Let $\cx$ be a doubling metric measure space.
The main purpose of this paper is to investigate behaviors of the
Lusin-area and $\gl$ functions on localized ${\rm BMO}$ spaces over
$\cx$, which is not necessary to be an
RD-space. So far, it is still not clear
whether the doubling property of $\cx$
is sufficient to guarantee the boundedness of
the Lusin-area function on these localized ${\rm BMO}$ spaces
over $\cx$. However, in this paper, when ${\mathcal X}$ has the $\dz$-annular decay
property for some $\dz\in (0,1]$ which was
introduced by Buckley in \cite{b99}, we
establish the boundedness of the Lusin-area
function, which is defined via kernels modeled on the semigroup
generated by the Schr\"odinger operator, from localized spaces ${\rm
BMO}_\rho({\mathcal X})$ to ${\rm BLO}_\rho({\mathcal X})$ without
invoking any regularity of considered kernels. The corresponding
boundedness of the $\gl$ function from ${\rm BMO}_\rho({\mathcal
X})$ to ${\rm BLO}_\rho({\mathcal X})$ is also obtained in this
paper. Moreover, an interesting phenomena is that unlike the
Lusin-area function, the boundedness of the $\gl$ function needs
neither the regularity of the kernels nor the $\dz$-annular decay
property of $\cx$, which reflects the difference between the
Lusin-area function and the $\gl$ function. These results are new
even on $\rd$ with the Lebesgue measure and the Heisenberg group,
and apply in a wide range of settings, for instance, to the
Schr\"odinger operator or the degenerate Schr\"odinger operator on
${{\mathbb R}}^d$, or the sub-Laplace Schr\"odinger operator on
Heisenberg groups or connected and simply connected nilpotent Lie
groups. Moreover, via some results on the pointwise multiplier of
${\rm bmo\,}(\rr)$ from \cite{ny85}, we construct a counterexample
to show that there exists a nonnegative function which is in ${\rm
bmo\,}(\rr)$ of Goldberg \cite{g79}, but not in ${\rm blo\,}(\rr)$
of \cite{hyy}. Thus, ${\rm blo\,}(\rr)\cap\{f\ge 0\}$ is a proper
subspace of ${\rm bmo\,}(\rr)$, which further indicates that our
above results on the boundedness
of the Lusin-area and $\gl$ functions even in $\rd$
with the Lebesgue measure or the Heisenberg
group also improve the existing results.

Moreover, motivated by Tessera \cite{te}, we introduce two
properties, for any metric space, the weak geodesic property and the
monotone geodesic property, which are slightly stronger variants of the
corresponding ones of Tessera \cite{te} (see Remark \ref{r4.1}
below) and are then proved to be respectively equivalent to the
chain ball property introduced by Buckley \cite{b99}. It was proved
by Buckley \cite{b99} that any length space, namely, the metric
space in which the distance between any pair of points equals the
infimum of the lengths of rectifiable paths joining them, has the
chain ball property and, for any metric space equipped with a
doubling measure, the chain ball property implies the
$\delta$-annular decay property for some $\dz\in (0,1]$. As an
application, we prove that any length space equipped with a doubling
measure has the weak geodesic property and hence the $\dz$-annular
decay property for some $\dz\in (0,1]$ without using the property of
John domains as in \cite{b99}.

This paper is organized as follows. Let $\cx$ be a doubling metric
measure space  and $\rho$ an admissible function on $\cx$. In
Section 2, we first recall the notions of the spaces
$\bmo_\rho(\cx)$ and $\blo_\rho(\cx)$. When $\cx=\rr$,
we construct a counterexample to show that there exists
a nonnegative function $f\in\sbmo(\rr)$, but $f\not\in{\rm
blo}(\rr)$; see Proposition \ref{p2.1} below.

In Section 3, if $\cx$ has the $\dz$-annular decay property for
some $\dz\in (0,1]$ and the Littlewood-Paley $g$-function $g(f)$
is bounded on $L^2(\cx)$, we
prove that if $f\in\mathrm{\,BMO}_\rho({\mathcal X})$, then
$[S(f)]^2\in\mathrm{\,BLO}_\rho({\mathcal X})$ with norm no more
than $C\|f\|^2_{\mathrm{\,BMO}_\rho({\mathcal X})}$, where $C$ is a
positive constant independent of $f$; see Theorem \ref{t3.1} below.
As a corollary, we obtain the boundedness of the Lusin-area function
from $\mathrm{\,BMO}_\rho({\mathcal X})$ to
$\mathrm{\,BLO}_\rho({\mathcal X})$; see Corollary \ref{c3.1} below.
The corresponding results for the $\gl$ function $\gl(f)$ are
established in Theorem \ref{t3.2} and Corollary \ref{c3.2} below,
where $\cx$ is not necessary to have the $\dz$-annular decay
property. Theorems \ref{t3.1} and \ref{t3.2} and Corollaries
\ref{c3.1} and \ref{c3.2} are true for the Schr\"odinger operator or
the degenerate Schr\"odinger operator on ${\mathbb R}^d$, or the
sub-Laplace Schr\"odinger operator on Heisenberg groups or connected
and simply connected nilpotent Lie groups. Moreover, for these
specific examples, it is known that the corresponding
Littlewood-Paley $g$-function is bounded on $L^2(\cx)$; see
\cite{yyz} for the detailed explanations.

We remark that the results obtained in Section 3 are also new even
on $\rd$ with the Lebesgue measure and the Heisenberg group, since
we do not need any regularity of involved kernels. However, to
establish the boundedness of the Lusin-area function on a doubling
metric measure space $\cx$, we need certain regularity of $\cx$,
namely, the $\dz$-annular decay property of $\cx$, which reflects
the speciality of the Lusin-area function, comparing with the
corresponding results of the $\gl$ function. Moreover, $\rd$ with
the Lebesgue measure and the Heisenberg group have the $\dz$-annular
decay property.

In Section 4, for any metric space, we introduce the notions of the
weak geodesic property and the monotone geodesic property in
Definition \ref{d4.1} below, which are proved respectively
equivalent to the chain ball property of Buckley in Theorem
\ref{t4.1} below. As an application of this result and \cite[Theorem
2.1]{b99}, we obtain in Corollary \ref{c4.1} below that for any
metric space equipped with a doubling measure, either the weak
geodesic property or the monotone geodesic property guarantees its
$\dz$-annular decay property for some $\dz\in (0,1]$. As an
application of Corollary \ref{c4.1}, we prove that any length space
equipped with a doubling measure has the $\dz$-annular decay property
for some $\dz\in(0,1]$; see Proposition \ref{p4.1} below.

Finally, we make some conventions. Throughout this paper, we always use
$C$ to denote a positive constant that is independent of the main
parameters involved but whose value may differ from line to line.
Constants with subscripts, such as $C_1$, do not change in
different occurrences. If $f\le Cg$, we then write $f\ls g$ or $g\gs
f$; and if $f \ls g\ls f$, we then write $f\sim g.$
We also use $B$ to denote a ball of $\cx$, and for $\lz>0$,
$\lz B$ denotes the ball with the same center as $B$, but radius
$\lz$ times the radius of $B$. Moreover, set
$B^\complement\equiv\cx\setminus B$. Also, for any set
$E\subset\cx$, $\chi_E$ denotes its characteristic function. For all
$f\in L^1_\loc(\cx)$ and balls $B$, we always set
$f_{B}\equiv\frac1{\mu(B)}\int_Bf(y)\,d\mu(y)$.

\section{The spaces $\bmo_\rho(\cx)$ and $\blo_\rho(\cx)$\label{s2}}

\hskip\parindent In this section, we first recall the notions of
localized $\bmo$ spaces over doubling metric measure spaces.
Moreover, visa some results on pointwise multipliers of ${\rm
bmo}(\rr)$, an example is constructed to show that there exists a
nonnegative function which is in ${\rm bmo}(\rr)$, but not in ${\rm
blo}(\rr)$.

We begin with the notions of doubling metric measure spaces
\cite{cw71,cw77} and admissible functions \cite{yz08}.

\begin{defn}\label{d2.1} \rm
Let $(\cx,\, d)$ be a metric space endowed with a regular Borel
measure $\mu$ such that all balls defined by $d$ have finite and
positive measure. For any $x\in \cx $ and $r\in(0, \fz)$, set the
ball $B(x,r)\equiv\{y\in \cx :\ d(x,y)<r\}.$ The triple
$(\cx,\,d,\,\mu)$ is called a doubling metric measure space if there
exists a constant $C_1\in[1, \fz)$ such that for all $x\in \cx $ and
$r\in(0, \fz)$,
$\mu(B(x, 2r))\le C_1\mu(B(x,r))\ (\mathrm{doubling\ property}).$
\end{defn}

From Definition \ref{d2.1}, it is easy to see that there exists
positive constants $C_2$ and $n$ such that for all
    $x\in \cx$, $r\in(0, \fz)$ and $\lz\in[1, \fz)$,
\begin{equation}\label{e2.1}
 \mu(B(x,\lz r))\le C_2\lz^ n\mu(B(x,r)).
\end{equation}

In what follows, we always let $\ol{B(x,\,r)}\equiv\{y\in \cx :\
d(x,y)\le r\}$, $V_r(x)\equiv\mu(B(x,\,r))$ and
$V(x,\,y)\equiv\mu(B(x,\,d(x,\,y)))$ for all $x,\,y\in\cx$ and
$r\in(0,\,\fz)$.

\begin{defn}[\cite{yz08}]\label{d2.2}\rm
A positive function $\rho$ on $\cx$ is called admissible if there
exist positive constants $C_0$ and $k_0$ such that for all
$x,\,y\in\cx$,
\begin{equation}\label{e2.2}
\frac1{\rho(x)} \le
C_0\frac1{\rho(y)}\lf(1+\frac{d(x,\,y)}{\rho(y)}\r)^{k_0}.
\end{equation}
\end{defn}

Obviously, if $\rho$ is a constant function, then $\rho$ is
admissible. Another non-trivial class of admissible functions is
given by the well-known reverse H\"older class $\mathcal\cb_q(\cx,
d, \mu)$ (see, for example \cite{ge73,m72,s95} for its definition on
$\rn$, and \cite{st89} for its definition on spaces of homogenous
type). Recall that a nonnegative potential $V$ is said to be in
$\mathcal\cb_q(\cx, d, \mu)$ (for short, $\cb_q(\cx)$) with
$q\in(1,\,\fz]$ if there exists a positive constant $C$ such that
for all balls $B$ of $\cx$,
$$\lf\{\frac1{|B|}\dint_B[V(y)]^q\,dy\r\}^{1/q}
\le \frac{C}{|B|}\dint_BV(y)\,dy$$ with the usual modification made
when $q=\fz$. It was proved in \cite[pp.\,8-9]{st89} that if $V\in
\cb_q(\cx)$ for some $q\in(1, \fz]$ and the measure $V(z)d\mu(z)$
has the doubling property, then $V$ is an ${\cal A}_p(\cx, d,
\mu)$-weight for some $p\in[1, \fz)$ in the sense of Muckenhoupt,
and also $V\in \cb_{q+\ez}(\cx)$ for some $\ez>0$. Here it should be
pointed out that, generally speaking, $V\in\cb_q(\cx)$ cannot guarantee the
doubling property of $V(z)d\mu(z)$, but when $\mu(B(x, r))$ is
continuous respect to $r$ for all $x\in\cx$ or $\cx$ has the
$\dz$-annular decay property (see Definition \ref{d3.1} below),
$V\in\cb_q(\cx)$ does imply the doubling property of $V(z)d\mu(z)$
by \cite[Theorem 17]{st89} or \cite[Proposition 3.7]{ma07},
respectively. Following \cite{s95}, for all $x\in\cx$, set
\begin{equation}\label{e2.3}
\rho(x)\equiv\sup\lf\{r>0:\hs \frac{r^2}
{\mu(B(x,\,r))}\dint_{B(x,\,r)}V(y)\,dy\le 1\r\};
\end{equation}
see also \cite{yz08}. It was proved in \cite[Proposition 2.1]{yz08}
that if the measure $V(z)d\mu(z)$ has the doubling property, then
$\rho$ in \eqref{e2.3} is an admissible function when $n\ge1$,
$q>\max\{1,\,n/2\}$ and $V\in \cb_q(\cx)$.

Now we recall the notions of the spaces $\bmo_\rho(\cx)$ and
$\blo_\rho(\cx)$ (see \cite{yyz}).

\begin{defn}[\cite{yyz}]\label{d2.3}\rm
 Let $\rho$ be an admissible function on $\cx$,
$\cd\equiv\{B(x,\,r)\subset\cx:\ x\in\cx,\ r\ge\rho(x)\}$ and
$q\in[1,\,\fz)$. A function $f\in L^q_{\loc}(\cx)$ is said to be in
the space $\bmo^q_\rho(\cx)$ if
$$\begin{array}{cl}
\|f\|_{\bmo_\rho^q(\cx)}\equiv&\dsup_{B\notin\cd}
\lf\{\frac1{\mu(B)}\int_B|f(y)-f_B|^qd\mu(y)\r\}^{1/q}\\
&+\dsup_{B\in\cd}
\lf\{\frac1{\mu(B)}\int_B|f(y)|^qd\mu(y)\r\}^{1/q}<\fz.
\end{array}$$
\end{defn}

\begin{rem}\label{r2.1}\rm
We denote $\bmo^1_\rho(\cx)$ simply by $\bmoz$. The space
$\bmo_\rho(\rd)$ when $\rho\equiv1$ was first introduced by Goldberg
\cite{g79}. If $q>\frac d2$, $V\in \cb_q(\rd)$ and $\rho$ is as in
\eqref{e2.3}, then $\bmo_\rho(\rd)$ is just the space
$\bmo_\cl(\rd)$ introduced by
 Dziub\'anski et al in \cite{dgmtz05}.
For all $q\in [1,\fz)$, $\bmo^q_\rho(\cx)\subsetneq \bmo(\cx)$.
\end{rem}

The following technical lemma is just Lemma 3.1 in \cite{yyz}.

\begin{lem}\label{l2.1}
Let $\rho$ be an admissible function on $\cx$ and $q\in[1,\,\fz)$.
Then $\bmoz=\bmo^q_\rho(\cx)$ with equivalent norms.
\end{lem}

\begin{defn}[\cite{yyz}]\label{d2.4}\rm
Let $\rho$ and $\cd$ be as in Definition \ref{d2.3} and
$q\in[1,\fz)$. A function $f\in L^q_{\loc}(\cx)$ is said to be in
the space $\blo_\rho(\cx)$ if
 \begin{eqnarray*}
\|f\|_{\blo^q_\rho(\cx)}&&\equiv \sup_{B\notin \cd}
\lf\{\frac1{\mu(B)}
\int_B\lf[f(y)-{\mathop\einf_B} f\r]^q\,d\mu(y)\r\}^{1/q}\\
&&\hs+\sup_{B\in\cd}\lf\{\frac1{\mu(B)}
\int_B|f(y)|^q\,d\mu(y)\r\}^{1/q}<\fz.
\end{eqnarray*}
\end{defn}

\begin{rem}\label{r2.2}\rm
(i) The space $\mathrm{BLO}(\rd)$ with the Lebesgue measure was
introduced by Coifman and Rochberg \cite{cr80}, and extended by
Jiang \cite{j05} to the setting of $\rd$ with a non-doubling
measure. The localized $\blo$ space was first introduced in
\cite{hyy} in the setting of $\rd$ with a non-doubling measure.

(ii) For all $q\in [1,\fz)$,
$\blo^q_\rho(\cx)\subset\bmo^q_\rho(\cx)$. We denote
$\blo^1_\rho(\cx)$ simply by $\bloz$.
\end{rem}

Even when $\rho\equiv1$, it is not so difficult to show that for all
$q\in [1, \fz)$, $\blo_\rho^q(\rd)$ is a proper subspace of
$\bmo_\rho^q(\rd)$. For example, if we set $f(x)\equiv
(\log|x|)\chi_{\{|x|\le1\}}(x)$ for all $x\in\rr$, then it is
easy to show that $f\in\bmo_1^q(\rr)$, but $f\not\in\blo_1^q(\rr)$. Notice
that the above function is non-positive. However, it is not so easy
to show that there exists a nonnegative function which is in
$\bmo_\rho^q(\rd)$, but not in $\blo_\rho^q(\rd)$.

Let $\cx=(\rr,\,|\cdot|,\,dx)$. Denote $\mathrm{BMO}_\rho(\rr)$ and
$\mathrm{BLO}_\rho(\rr)$ with $\rho\equiv1$, respectively, by
$\sbmo(\rr)$ and $\sblo(\rr)$. In the rest of this section, we
construct the following interesting counterexample.

\begin{prop}\label{p2.1}
There exists a nonnegative function $f\in\sbmo(\rr)$, but
$f\not\in\sblo(\rr)$.
\end{prop}

We first recall some notation and notions. Let $\phi$ be a positive
non-decreasing function on $(0,\,\fz)$. Define
\begin{equation*}
{\rm BMO}^\phi(\rr)\equiv\lf\{f\in L^1_{\loc}(\rr):\,\sup_{{\rm
balls}\,B\subset\rr} \frac{\mathrm{MO}(f,\,B)}{\phi(r_B)}<\fz\r\}
\end{equation*}
and
\begin{equation*} \wz{{\rm BMO}^\phi}(\rr)\equiv\lf\{f\in
L^1_{\loc}(\rr):\,|f_{B(0,\,1)}|+\sup_{{\rm balls}\,B\subset\rr}
\frac{\mathrm{MO}(f,\,B)}{\phi(r_B)}<\fz\r\},
\end{equation*} where $\mathrm{MO}(f,\,B)
=\frac{1}{|B|}\int_B|f(x)-f_B|\,dx$ and $r_B$ denotes the radius of
ball $B$. Recall that $f_B=\frac1{|B|}\int_B f(y)\,dy$. Then ${\rm
BMO}^\phi(\rr)$ modulo constants is a Banach space, but $\wz{{\rm
BMO}^\phi}(\rr)$ is itself a Banach space modulo null-functions; see
\cite{ny85}.

The following conclusion is just Lemma 2.2 in \cite{ny85}.

\begin{lem}\label{l2.2} If $|F(x)-F(y)|\le
C|x-y|$, then $\mathrm{MO}(F(f),\,B)\le 2C\,\mathrm{MO}(f,\,B)$.
\end{lem}

For a positive non-decreasing function $\phi$ on $(0,\,\fz)$, we
define strictly positive functions $\Phi^*(r)$ and $\Phi_*(r)$ by
setting
\begin{equation*}
\Phi^*(r)\equiv\lf\{\begin{array}{cl}\dint_1^r\frac{\phi(t)}t\,dt,&\quad{\rm
if}\quad2\le
r;\\[4mm]\dint_1^2\frac{\phi(t)}t\,dt,&\quad{\rm
if}\quad0<r<2,\end{array}\r. \quad{\rm and}\quad
\Phi_*(r)\equiv\lf\{\begin{array}{cl}\dint_r^2\frac{\phi(t)}t\,dt,&\quad{\rm
if}\quad0<r\le
1;\\[4mm]\dint_1^2\frac{\phi(t)}t\,dt,&\quad{\rm if}\quad1<r.\end{array}\r.
\end{equation*}

The following result is just Lemma 2.4 in \cite{ny85}.

\begin{lem}\label{l2.3}
Assume that $\frac{\phi(t)}t$ is almost decreasing. Then
$\Phi^*(|x|),\,\Phi_*(|x|)\in\wz{{\rm BMO}^\phi}(\rr).$
\end{lem}

Recall that a function $g$ on $\rr$ is called a pointwise multiplier
on $\sbmo(\rr)$, if the pointwise multiplication $fg$ belongs to
$\sbmo(\rr)$ for all $f\in\sbmo(\rr)$.

Set
\begin{equation}\label{e2.4}
\psi(r)=\lf[\int^2_{\min\{1,\,r\}}\frac{1}{t}dt\r]^{-1}
\quad\text{for}\quad r\in (0,\,\infty).
\end{equation}
Then $\psi$ is increasing and $\frac{\psi(t)}{t}$ is almost
decreasing. The following Lemma \ref{l2.4} is a special case of
Theorem 3 in \cite{ny85}.

\begin{lem}\label{l2.4}
A function $g$ on $\rr$ is a pointwise multiplier on $\sbmo(\rr)$ if
and only if $g\in\wz{\rm BMO^\psi}(\rr)\cap L^\fz(\rr)$, where
$\psi$ is as in \eqref{e2.4}.
\end{lem}

Then we have the following conclusion.

\begin{prop}\label{p2.2}
Let $\psi$ be as in \eqref{e2.4}. Set
\begin{equation*}
\Psi_*(r)\equiv
\lf\{\begin{array}{cl}\dint_r^2\frac{\psi(t)}t\,dt,&\quad{if}\quad0<r\le
1;\\[4mm]\dint_1^2\frac{\psi(t)}t\,dt,&\quad{if}\quad1<r,\end{array}\r.
\end{equation*}
and
\begin{equation}\label{e2.5}
g(x)\equiv{\sin\Psi_*(|x|)}\quad {for}\,\,x\in\rr.
\end{equation}
Then $g$ is a pointwise multiplier on $\sbmo(\rr)$.
\end{prop}

\begin{proof}\rm
By Lemma \ref{l2.4}, we only need to prove that
$\sin\Psi_*(|x|)\in\wz{\rm BMO^\psi}(\rr)\cap L^\fz(\rr)$. From
Lemma \ref{l2.3}, it follows that $\Psi_*(|x|)\in\wz{\rm
BMO^\psi}(\rr)$, which via Lemma \ref{l2.2} shows that
$\sin\Psi_*(|x|)\in\wz{\rm BMO^\psi}(\rr)$. Obviously,
$\sin\Psi_*(|x|)\in L^\fz(\rr)$, which completes the proof of
Proposition \ref{p2.2}.
\end{proof}

Now we prove Proposition \ref{p2.1}.

\begin{proof}[Proof of Proposition \ref{p2.1}]
Let $g$ be as in \eqref{e2.5}. For $x\in\rr$, set
\begin{equation*}
f(x)\equiv\lf\{\begin{array}{ll}\log(2/|x|),&\quad{\rm if}\quad|x|\le2;\\
[2mm]0,&\quad{\rm if}\quad|x|>2.\end{array}\r.
\end{equation*}

Then we shall show $|fg|\in\sbmo(\rr)$, but $|fg|\not\in\sblo(\rr)$.

It is obvious that $f\in\sbmo(\rr)$. Since $g$ is a pointwise
multiplier on $\sbmo(\rr)$, $fg\in\sbmo(\rr)$, and so
$|fg|\in\sbmo(\rr)$.

Now we turn our attention to prove that $|fg|\not\in{\rm blo}(\rr)$.
Notice that
\begin{equation*}
\Psi_*(r)\equiv\lf\{\begin{array}{ll}
1+\dint_r^1\frac{dt}{t\log(2/t)},&\quad{\rm if}\quad0<r\le1;\\[4mm]
1,&\quad{\rm if}\quad1<r.\end{array}\r.
\end{equation*}
So
\begin{equation*}
g(x)=\sin\left(1+\int^1_{|x|}\frac{dt}{t\log(2/t)}\right),\quad\quad
{\rm if}\quad |x|\le 1.
\end{equation*}
For $k=2,\,3,\,4,\,\cdots$, choose $r_k>0$ such that
\begin{equation*}
\Psi_*(r_k)=1+\int^1_{r_k}\frac{dt}{t\log(2/t)}=\frac\pi4k.
\end{equation*}
Then $1>r_2>r_3>r_4>\cdots$, and $r_k\rightarrow0$ as
$k\rightarrow\fz$. Let $m\in\nn$. For $x\in[r_{8m+4},\,r_{8m+3})$,
we have $\Psi_*(x)\in((2m+\frac34)\pi,\,(2m+1)\pi]$, which implies
that $\sin\Psi_*(x)\ge0$, $\cos\Psi_*(x)<0$ and
$\sin\Psi_*(x)+\cos\Psi_*(x)<0.$
Then we have the following:
$$\begin{array}{|c||c|c|c|c|c|}
\hline
  x   & \ \cdots\ & \ r_{8m+4}\ & \ \cdots\ & r_{8m+3}  & \ \cdots\ \\
\hline
  (fg)'(x) & \ \cdots\      & +           & +         & 0 & \ \cdots\   \\
\hline
  (fg)''(x) & \ \cdots\      & -           & -         & - & \ \cdots\  \\
\hline
  fg(x)     & \ \cdots\     & 0           & \nearrow  & \rule{0mm}{5mm}
  \frac{\sqrt{2}\log(2/r_{8m+3})}{2} & \ \cdots\   \\
\hline
\end{array}$$
In fact, for $x\in(r_{8m+4},\,r_{8m+3})$,
\begin{align*}
(fg)'(x)&=\lf(-\frac1x\r)\sin\Psi_*(x)+[\log(2/x)]\cos\Psi_*(x)
\Big(-\frac{1}{x\log(2/x)}\Big)\\
&=-\frac1x \left(\sin\Psi_*(x)+\cos\Psi_*(x)\right)>0,
\end{align*}
and
\begin{equation*}
(fg)''(x)=\frac1{x^2}\left(\sin\Psi_*(x)+\cos\Psi_*(x)\right)
-\frac1x\left(\cos\Psi_*(x)-\sin\Psi_*(x)\right)[\Psi_*(x)]'<0.
\end{equation*}
Hence $fg$ is nonnegative, increasing and strictly concave on
$[r_{8m+4},\,r_{8m+3})$, and so
\begin{align*}
&\frac{1}{r_{8m+3}-r_{8m+4}}\int_{r_{8m+4}}^{r_{8m+3}}[|fg|(x)-\mathop\einf(|fg|)]dx\\
&\hs=\frac{1}{r_{8m+3}-r_{8m+4}}\int_{r_{8m+4}}^{r_{8m+3}}fg(x)\,dx\\
&\hs\ge\frac{1}{2}\frac{\sqrt{2}\log(2/r_{8m+3})}{2}\rightarrow\fz\quad{\rm
as}\quad m\rightarrow\fz,
\end{align*}
which implies that $|fg|\not\in\sblo(\rr)$. This finishes the proof
of Proposition \ref{p2.1}.
\end{proof}

\section{Boundedness of Lusin-area and
$\gl$ functions\label{s3} }

\hskip\parindent Let $\rho$ be an admissible function and $\cx$ a
doubling metric measure space. In this section, we consider the
boundedness of certain variant of Lusin-area and $\gl$ functions
from $\bmo_\rho(\cx)$ to $\blo_\rho(\cx)$. We remark that unlike the
boundedness of the $\gl$ function, to obtain the boundedness of the
Lusin-area function, we need to assume that $\cx$ has the
$\dz$-annular decay property. Several remarks on this property are
given in Section \ref{s4}.

\begin{defn}\label{d3.1}\rm For $\dz\in(0, 1]$ and a doubling
metric measure space $(\cx, d, \mu)$, $(\cx, d, \mu)$ is said to
have the {\it $\dz$-annular decay property} if there exists a
constant $K\in [1, \fz)$ such that for all $x\in\cx$, $s\in (0,\fz)$
and $r\in (s, \fz)$,
\begin{equation}\label{e3.1}
\mu(B(x,\,r+s))-\mu(B(x,\,r))\le K\lf(\frac
sr\r)^{\dz}\mu(B(x,\,r)).
\end{equation}
\end{defn}

Observe that if $r\in (0,s]$, then \eqref{e3.1} is a simple
conclusion of the doubling property \eqref{e2.1} of $\mu$.

Let $\rho$ be an admissible function on $\cx$ and $\{Q_t\}_{t>0}$ a
family of  operators bounded on $\lt$ with  integral kernels
$\{Q_t(x,\,y)\}_{t>0}$ satisfying that there exist constants
$C,\,\dz_1\in(0,\,\fz)$, $\dz_2\in(0,\,1)$ and $\gz\in(0,\,\fz)$
such that for all $t\in(0,\,\fz)$ and $x,\,y\in\cx$,

$(Q)_{\rm i}$ $|Q_t(x,\,y)|\le C\frac1{V_t(x)+V(x,\,y)}(\frac
t{t+d(x,\,y)})^{\gz}(\frac {\rho(x)}{t+\rho(x)})^{\dz_1}$;

$(Q)_{\rm ii}$ $|\int_\cx Q_t(x,\,z)\,d\mu(z)|\le C(\frac
t{t+\rho(x)})^{\dz_2}$.

For all $f\in L^1_\loc(\cx)$ and $x\in\cx$, define the
Littlewood-Paley $g$-function by setting
\begin{equation}\label{e3.2}
g(f)(x)\equiv\lf\{\int_0^\fz|Q_t(f)(x)|^2\,\frac{dt}{t}\r\}^{1/2},
\end{equation}
and Lusin-area and $\gl$ functions, respectively, by setting
\begin{equation}\label{e3.3}
S(f)(x)\equiv\lf\{\int_0^\fz\int_{d(x,\,y)<t} |Q_t(f)(y)|^2 \frac
{d\mu(y)}{V_t(y)}\frac {dt}{t}\r\}^{1/2},
\end{equation}
and
\begin{equation}\label{e3.4}
g^*_\lz(f)(x)\equiv\lf\{\iint_{\cx\times(0,\,\fz)}\lf(\frac{t}{t+d(x,\,y)}\r)^\lz
|Q_t(f)(y)|^2 \frac {d\mu(y)}{V_t(y)}\frac {dt}{t}\r\}^{1/2},
\end{equation}
where $\lz\in(0,\,\fz)$.

We first have the following technical lemma.

\begin{lem}\label{l3.1}
Assume that the Littlewood-Paley $g$-function $g(f)$ in \eqref{e3.2}
is bounded on $L^2(\cx)$. Then the Lusin-area function $S(f)$ in
\eqref{e3.3} and the $\gl$ function $\gl(f)$ in \eqref{e3.4} with
$\lz\in(n,\,\fz)$ are bounded on $L^2(\cx)$, where $n$ is as in
\eqref{e2.1}.
\end{lem}

\begin{proof}\rm Since for all $x\in\cx$, $S(f)(x)\le\gl(f)(x)$. We
only need to prove the $L^2(\cx)$-boundedness of $\gl(f)$.

To this end, we have
\begin{align*}
\dint_{\cx}\Big[\gl(f)(x)\Big]^2\,d\mu(x)
&=\dint_{\cx}\iint_{\cx\times(0,\,\fz)}
\bigg(\frac{t}{t+d(x,\,y)}\bigg)^{\lz}|Q_t(f)(y)|^2
\frac{d\mu(y)}{V_t(y)}\frac{dt}{t}\,d\mu(x)\\
&\le\dint_{\cx}\dint_0^{\fz}|Q_t(f)(y)|^2\dfrac{dt}{t}\dsup_{t>0}
\bigg[\dint_{\cx}\bigg(\frac{t}{t+d(x,\,y)}\bigg)^{\lz}
\dfrac{1}{V_t(y)}\,d\mu(x)\bigg]\,d\mu(y)\\
&=\dint_{\cx}[g(f)(y)]^2
\dsup_{t>0}\bigg[\dint_{\cx}\bigg(\frac{t}{t+d(x,\,y)}\bigg)^{\lz}
\dfrac{1}{V_t(y)}\,d\mu(x)\bigg]\,d\mu(y).
\end{align*}
Moreover, for all $y\in\cx$ and $t>0$, we obtain
\begin{align*}
&\dint_{\cx}\lf(\frac{t}{t+d(x,\,y)}\r)^{\lz}
\dfrac{1}{V_t(y)}\,d\mu(x)\\
&\hs=\dint_{d(x,\,y)< t}\lf(\frac{t}{t+d(x,\,y)}\r)^{\lz}
\dfrac{1}{V_t(y)}\,d\mu(x)+\dint_{d(x,\,y)\ge
t}\cdots\\
&\hs\ls1+\dsum_{k=0}^{\fz}\dint_{2^kt\le d(x,\,y)<2^{k+1}t}
\lf(\frac{t}{t+d(x,\,y)}\r)^{\lz} \dfrac{1}{V_t(y)}\,d\mu(x)
\ls1+\dsum_{k=0}^{\fz}2^{-k(\lz-n)}\ls1,
\end{align*}
where we used the assumption that $\lz\in(n,\,\fz)$. Thus,
$\|\gl(f)\|_{L^2(\cx)}\ls\|g(f)\|_{L^2(\cx)}$, which completes the
proof of Lemma \ref{l3.1}.
\end{proof}

\begin{thm}\label{t3.1}\rm
Let $\cx$ be a doubling metric measure space having the
$\dz$-annular decay property for some $\dz\in(0,\,1]$. Let $\rho$ be
an admissible function on $\cx$ and the Lusin-area function $S(f)$
as in \eqref{e3.3}. Assume that the Littlewood-Paley $g$-function in
\eqref{e3.2} is bounded on $L^2(\cx)$. Then there exists a positive
constant $C$ such that for all $f\in\bmo_{\rho}(\cx)$,
$[S(f)]^2\in\blo_{\rho}(\cx)$ and
$\|[S(f)]^2\|_{\blo_{\rho}(\cx)}\le C\|f\|^2_{\bmo_{\rho}(\cx)}$.
\end{thm}

\begin{proof}\rm By the homogeneity of
$\|\cdot\|_\bmoz$ and $\|\cdot\|_\bloz$, we may assume that
$f\in\bmo_{\rho}(\cx)$ and $\|f\|_{\bmo_{\rho}(\cx)}=1.$ Let
$B\equiv B(x_0,\,r)$. We prove Theorem \ref{t3.1} by considering the
following two cases. First, we notice that the
$L^2(\cx)$-boundedness of $g$ via Lemma \ref{l3.1} implies that
$S(f)$ is bounded on $L^2(\cx)$.

{\bf Case I.}\quad $r\ge \rho(x_0)$. In this case, we prove that
\begin{equation}\label{e3.5}
\frac1{\mu(B)}\int_B[S(f)(x)]^2\,d\mu(x)\ls1.
\end{equation}
For any $x\in B$, write
\begin{align*}
[S(f)(x)]^2&=\int_0^{8\rho(x)}\int_{d(x,\,y)<t} |Q_t(f)(y)|^2 \frac
{d\mu(y)}{V_t(y)}\frac {dt}{t}+\int_{8\rho(x)}^\fz\int_{d(x,\,y)<t}
\cdots\\
&\equiv[S_1(f)(x)]^2+[S_2(f)(x)]^2.
\end{align*}

By the $L^2(\cx)$-boundedness of $S(f)$, \eqref{e2.1} and Lemma
\ref{l2.1}, we have
\begin{equation}\label{e3.6}
\frac1{\mu(B)}\int_B[S_1(f\chi_{2B})(x)]^2\,d\mu(x)\ls
\frac1{\mu(B)}\int_{2B}|f(x)|^2\,d\mu(x)\ls1.
\end{equation}

Fix $x\in B$. Notice that if $d(x,\,y)<t$, then
\begin{equation}\label{e3.7}
t+d(y,\,z)\sim t+d(x,\,z)\quad{\rm and}\quad V_t(y)+V(y,\,z)\sim
V_t(x)+V(x,\,z).
\end{equation}
From \eqref{e3.7} and $(Q)_{\rm i}$, it follows that for all
$y\in\cx$ with $d(x,\,y)<t$, we have
\begin{align}\label{e3.8}
|Q_t(f\chi_{(2B)^\complement})(y)|&\ls \int_{(2B)^\complement}
\frac{1}{V_t(y)+V(y,\,z)}\Big(\frac{t}{t+d(y,\,z)}\Big)^\gz
|f(z)|\,d\mu(z)\\
&\ls\int_{(2B)^\complement}
\frac{1}{V_t(x)+V(x,\,z)}\Big(\frac{t}{t+d(x,\,z)}\Big)^\gz
|f(z)|\,d\mu(z)\nonumber\\
&\ls\Big(\frac{t}{r}\Big)^\gz\sum_{j=1}^\fz
\frac{2^{-j\gz}}{V_{2^{j-1}r}(x)}\int_{d(x,\,z)<2^jr}
|f(z)|\,d\mu(z)\ls\Big(\frac{t}{r}\Big)^\gz.\nonumber
\end{align}
Observe that by \eqref{e2.2}, for any $a\in(0, \fz)$, there exists a
constant $\wz C_a\in[1, \fz)$ such that for all $x$, $y\in\cx$ with
$d(x, y)\le a\rho(x)$,
\begin{equation}\label{e3.9}
\rho(y)/\wz {C}_a\le\rho(x)\le {\wz C}_a\rho(y).
\end{equation}
By this and $r\ge \rho(x_0)$, we obtain that for all $x\in B$,
$\rho(x)\ls r$. Notice that for all $x,\,y\in\cx$ satisfying
$d(x,\,y)<t$, we have
\begin{equation}\label{e3.10}
V_t(x)\sim V_t(y).
\end{equation}
It then follows from \eqref{e3.8} and \eqref{e3.10} together with
$\gz\in(0,\,\fz)$ that
\begin{equation}\label{e3.11}
\frac{1}{\mu(B)}\int_B[S_1(f\chi_{(2B)^\complement})(x)]^2\,d\mu(x)
\ls\frac{1}{\mu(B)}\int_B\int_0^{8\rho(x)}
\Big(\frac{t}{r}\Big)^{2\gz}\frac{dt}{t}\,d\mu(x)\ls1,
\end{equation}
which together with \eqref{e3.6} tells us that
\begin{equation}\label{e3.12}
\frac{1}{\mu(B)}\int_B[S_1(f)(x)]^2\,d\mu(x)\ls1.
\end{equation}

Fix $x\in B$. Notice that for all $y\in\cx$ with $d(x,\,y)<t$ and
$t\ge8\rho(x)$, by \eqref{e2.2}, we have
\begin{equation}\label{e3.13}
\frac{\rho(y)}{t+\rho(y)}\ls\lf(\frac{\rho(x)}{t}\r)^{\frac1{1+k_0}}.
\end{equation}
From \eqref{e3.7}, \eqref{e3.13} and $(Q)_{\rm i}$, it follows that
\begin{align}\label{e3.14}
|Q_t(f)(y)|&\le\int_{\cx}\frac{1}{V_t(y)+V(y,\,z)}
\Big(\frac{t}{t+d(y,\,z)}\Big)^\gz
\Big(\frac{\rho(y)}{t+\rho(y)}\Big)^{\dz_1}
|f(z)|\,d\mu(z)\\
&\ls\int_{\cx}
\frac{1}{V_t(x)+V(x,\,z)}\Big(\frac{t}{t+d(x,\,z)}\Big)^\gz
\Big(\frac{\rho(x)}t\Big)^{\frac{\dz_1}{1+k_0}}|f(z)|\,d\mu(z)\nonumber\\
&\ls\Big(\frac{\rho(x)}{t}\Big)^{\frac{\dz_1}{1+k_0}}
\sum_{j=0}^\fz\frac{2^{-j\gz}}{V_{2^{j-1}t}(x)}\int_{d(x,\,z)<2^jt}
|f(z)|\,d\mu(z)\ls\Big(\frac{\rho(x)}{t}\Big)^{\frac{\dz_1}{1+k_0}}.\nonumber
\end{align}
Thus,
$$
\frac{1}{\mu(B)}\int_B[S_2(f)(x)]^2\,d\mu(x)\ls\frac{1}{\mu(B)}
\int_B\int_{8\rho(x)}^\fz\Big(\frac{\rho(x)}{t}\Big)^{\frac{2\dz_1}{1+k_0}}
\frac{dt}{t}\,d\mu(x)\ls1,
$$
which along with \eqref{e3.12} yields \eqref{e3.5}. Moreover, the
fact that \eqref{e3.5} holds for all balls $B(x_0,\,r)$ with
$r\ge\rho(x_0)$ tells us that $S(f)(x)<\fz$ for almost every
$x\in\cx$.

{\bf Case II.}\quad $r<\rho(x_0)$. In this case, if
$r\ge\rho(x_0)/8$, then by \eqref{e2.1} and \eqref{e3.5}, we have
\begin{equation*}
\frac1{\mu(B)}\int_B\lf\{[S(f)(x)]^2-{\mathop\einf_B}[S(f)]^2\r\}\,d\mu(x)
\ls \frac1{\mu(8 B)}\dint_{8 B}[S(f)(x)]^2\,d\mu(x)\ls1,
\end{equation*}
which is desired. If $r<\rho(x_0)/8$, it suffices to prove that for
$\mu$-almost every $y\in B$,
\begin{equation*}
\frac1{\mu(B)}\int_B\lf\{[S(f)(x)]^2-[S(f)(y)]^2\r\}\,d\mu(x)\ls1.
\end{equation*}
For all $x\in B$, write
\begin{eqnarray*}
[S(f)(x)]^2&&=\int_0^{8 r}\int_{d(x,\,y)<t}
|Q_t(f)(y)|^2\frac{d\mu(y)}{V_t(y)}\frac{dt}t
+\int_{8 r}^{8\rho(x_0)}\cdots+\int_{8\rho(x_0)}^\fz\cdots\\
&&\equiv[S_r(f)(x)]^2+[S_{r,\,x_0}(f)(x)]^2+[S_\fz(f)(x)]^2.
\end{eqnarray*}
Observe that for $\mu$-almost every $y\in B$,
\begin{eqnarray*}
&&\frac1{\mu(B)}\int_B\lf\{[S(f)(x)]^2-[S(f)(y)]^2\r\}\,d\mu(x)\\
&&\hs\le \frac1{\mu(B)}\int_B\lf\{[S_r(f)(x)]^2+[S_\fz(f)(x)]^2+
[S_{r,\,x_0}(f)(x)]^2-[S_{r,\,x_0}(f)(y)]^2 \r\}\,d\mu(x).
\end{eqnarray*}\

We first prove that
\begin{equation}\label{e3.15}
\frac1{\mu(B)}\int_B[S_r(f)(x)]^2\,d\mu(x)\ls1.
\end{equation}
Write $f\equiv f_1+f_2+f_B,$ where $f_1\equiv(f-f_B)\chi_{2B}$ and
$f_2\equiv(f-f_B)\chi_{(2B)^\complement}$. By the
$L^2(\cx)$-boundedness of $S(f)$, \eqref{e2.1} and Lemma \ref{l2.1},
we have
\begin{equation}\label{e3.16}
\frac1{\mu(B)}\int_B[S_r(f_1)(x)]^2\,d\mu(x) \ls
\frac1{\mu(B)}\int_{2B}|f-f_B|^2\,d\mu(x)\ls1.
\end{equation}

Fix $x\in B$. Then for all $y\in\cx$ with $d(x,\,y)<t$, by ${\rm
(Q)_{i}}$, \eqref{e3.7}, \eqref{e2.1} and the fact that
$|f_{2^{j+1}B}-f_B|\ls j$ for all $j\in\nn$, we have
\begin{align*}
|Q_t(f_2)(y)|&\le\int_{(2B)^\complement}
\frac1{V_t(y)+V(y,\,z)}\lf(\frac{t}{t+d(y,\,z)}\r)^{\gz}|f(z)-f_B|\,d\mu(z)\\
&\ls\int_{(2B)^\complement}
\frac1{V_t(x)+V(x,\,z)}\lf(\frac{t}{t+d(x,\,z)}\r)^{\gz}
|f(z)-f_B|\,d\mu(z)\\
&\ls\sum_{j=1}^\fz \lf(\frac
t{2^{j-1}r}\r)^{\gz}\lf[\frac1{V_{2^{j-1}r}(x)}
\int_{2^{j+1}B}\lf[|f(z)-f_{2^{j+1}B}|+|f_{2^{j+1}B}-f_B|\r]
\,d\mu(z)\r]\nonumber\\
&\ls\lf(\frac tr\r)^{\gz}\sum_{j=1}^\fz j2^{-j\gz} \ls\lf(\frac
tr\r)^{\gz},\nonumber
\end{align*}
which together with \eqref{e3.10} leads to that
\begin{equation*}
\frac1{\mu(B)}\int_B[S_r(f_2)(x)]^2\,d\mu(x)\ls \int_0^{8
r}\lf(\frac tr\r)^{2\gz} \,\frac{dt}t\ls1.
\end{equation*}

By this and \eqref{e3.16}, to prove \eqref{e3.15}, it remains to
show that
\begin{equation}\label{e3.17}
\frac1{\mu(B)}\int_B[S_r(f_B)(x)]^2\,d\mu(x)\ls1.
\end{equation}
Let $k$ be the smallest positive integer satisfying
$2^kr\ge\rho(x_0)$. Then,
\begin{equation}\label{e3.18}
|f_B|\le|f_B-f_{2B}|+|f_{2B}-f_{2^2B}|+\cdots
+|f_{2^{k-1}B}-f_{2^kB}|+|f_{2^kB}|\ls\log\frac{\rho(x_0)}{r}.
\end{equation}
On the other hand, fix $x\in B(x_0,\,r)$ with $r<\rho(x_0)/8$. Then
for all $y\in\cx$ satisfying $d(x,\,y)<t$ with $t\in(0,\,8 r)$, by
\eqref{e3.9}, we have $\rho(y)\sim\rho(x_0)$. Hence, by $(Q)_{\rm
ii}$ and \eqref{e3.18}, we have
$$|Q_t(f_B)(y)|\ls \lf(\frac t{\rho(y)}\r)^{\dz_2}|f_B|\ls \lf(\frac
t{\rho(x_0)}\r)^{\dz_2} \log\frac{\rho(x_0)}{r},$$ which via $t\le 8
r<\rho(x_0)$ further yields \eqref{e3.17}.

Now we turn our attention to prove that
\begin{equation}\label{e3.19}
\frac1{\mu(B)}\int_B[S_\fz(f)(x)]^2\,d\mu(x)\ls1.
\end{equation}
Fix $x\in B(x_0,\,r)$. Let $a\in[1/8, \fz)$ and ${\wz C}_a$ be as in
\eqref{e3.9}. We first prove that for all $f\in \bmoz$ with
$\|f\|_\bmoz=1$, $y\in\cx$ with $d(x,\,y)<t$ and $t\le 8{\wz
C}_a\rho(x_0)$,
\begin{equation}\label{e3.20}
|Q_t(f)(y)|\ls 1.
\end{equation}
 In fact, by ${\rm (Q)_{i}}$ and \eqref{e3.7}, we obtain
\begin{eqnarray}\label{e3.21}
\quad |Q_t(f-f_{B(x,\,t)})(y)|
&&\le\int_{\cx}\frac1{V_t(y)+V(y,\,z)}
\lf(\frac{t}{t+d(y,\,z)}\r)^{\gz}
|f(z)-f_{B(x,\,t)}|\,d\mu(z)\\
&&\ls\int_{\cx}\frac1{V_t(x)+V(x,\,z)}
\lf(\frac{t}{t+d(x,\,z)}\r)^{\gz}
|f(z)-f_{B(x,\,t)}|\,d\mu(z)\nonumber\\
&&\ls\sum_{j=0}^\fz 2^{-j\gz}\frac1{V_{2^{j-1}t}(x)}
\int_{d(x,\,z)<2^jt}|f(z)-f_{B(x,\,t)}|\,d\mu(z)\ls1.\nonumber
\end{eqnarray}
It follows from \eqref{e3.9} that for all $y\in\cx$ with
$d(x,\,y)<t\le8{\wz C}_a\rho(x_0)$,
$\rho(y)\sim\rho(x_0)\sim\rho(x)$, which together with the fact that
for all $x\in\cx$,
$|f_{B(x,\,t)}|\le|f_{B(x,\,t)}-f_{B(x,\,\rho(x))}|+|f_{B(x,\,\rho(x))}|\ls
1+\log\frac{\rho(x)}{t}$ (by \eqref{e3.18}), and $(Q)_{\rm ii}$
shows that
\begin{eqnarray*}
|Q_t(f_{B(x,\,t)})(y)|\ls\lf(\frac t{\rho(y)}\r)^{\dz_2}
\lf(1+\log\frac{\rho(x)}{t}\r)\ls\lf(\frac t{\rho(x)}\r)^{\dz_2}
\lf(1+\log\frac{\rho(x)}{t}\r)\ls1.
\end{eqnarray*}
Combining this and \eqref{e3.21} proves \eqref{e3.20}.

Using \eqref{e3.20}, \eqref{e3.9}, \eqref{e3.10} and \eqref{e3.14},
we have that for all $x\in B$,
\begin{eqnarray*}
&&\int_{8\rho(x_0)}^\fz\int_{d(x,\,y)<t}|Q_t(f)(y)|^2\,
\frac{d\mu(y)}{V_t(y)}\frac{dt}t\\
&&\hs\le\int_{8\rho(x_0)}^{8{\wz
C}_a\rho(x_0)}\int_{d(x,\,y)<t}|Q_t(f)(y)|^2\,
\frac{d\mu(y)}{V_t(y)}\frac{dt}t+\int_{8{\wz
C}_a\rho(x_0)}^\fz\cdots\\
&&\hs\ls1+\int_{8{\wz
C}_a\rho(x_0)}^\fz\lf(\frac{\rho(x)}{t}\r)^{\frac{2\dz_1}{1+k_0}}
\,\frac{dt}t\ls1,
\end{eqnarray*}
which yields \eqref{e3.19}.

By \eqref{e3.15} and \eqref{e3.19}, we reduce the proof of Theorem
\ref{t3.1} to show that for $\mu$-almost every $x'\in B$,
\begin{equation}\label{e3.22}
\frac1{\mu(B)}\int_B\lf\{[S_{r,\,x_0}(f)(x)]^2-[S_{r,\,x_0}(f)(x')]^2
\r\}\,d\mu(x)\ls1.
\end{equation}
For any $x,\,x'\in B$ such that $S_{r,\,x_0}(f)(x)$ and
$S_{r,\,x_0}(f)(x')$ are finite, write
\begin{align*}
&[S_{r,\,x_0}(f)(x)]^2-[S_{r,\,x_0}(f)(x')]^2\\
&\hs=\int_{8
r}^{8\rho(x_0)}\int_{d(x,\,y)<t}|Q_t(f)(y)|^2\frac{d\mu(y)}{V_t(y)}\frac{dt}{t}-
\int_{8
r}^{8\rho(x_0)}\int_{d(x',\,y)<t}\cdots\\
&\hs\le\int_{8 r}^{8\rho(x_0)}\int_{B(x,\,t)\triangle
B(x',\,t)}|Q_t(f-f_B)(y)|^2\frac{d\mu(y)}{V_t(y)}\frac{dt}{t}\\
&\hs\hs+\int_{8 r}^{8\rho(x_0)}\int_{B(x,\,t)\triangle
B(x',\,t)}|Q_t(f_B)(y)|^2\frac{d\mu(y)}{V_t(y)}\frac{dt}{t}\equiv{\rm
J_1}+{\rm J_2},
\end{align*}
where $B(x,\,t)\triangle B(x',\,t)\equiv[B(x,\,t)\setminus
B(x',\,t)]\bigcup\,[B(x',\,t)\setminus B(x,\,t)]$.

By the facts that $x,\,x'\in B$ and $t\ge8 r$, we have
$B(x,\,t-2r)\subset[B(x,\,t)\bigcap B(x',\,t)]$. Since $\cx$ has the
$\dz$-annular decay property for some $\dz\in(0,\,1]$, we obtain
$$\mu(B(x,\,t)\setminus
B(x',\,t))\le\mu(B(x,\,t))-\mu(B(x,\,t-2r))
\ls\Big(\frac{r}{t}\Big)^{\dz}\mu(B(x,\,t)).$$ By symmetry, we also
have $\mu(B(x',\,t)\setminus
B(x,\,t))\ls\big(\frac{r}{t}\big)^{\dz}\mu(B(x',\,t))$, which
together with \eqref{e2.1} implies that
\begin{equation}\label{e3.23}
\mu(B(x,\,t)\triangle
B(x',\,t))\ls\Big(\frac{r}{t}\Big)^{\dz}\mu(B(x,\,t)).
\end{equation}
By $(Q)_{\rm i}$, \eqref{e3.7}, \eqref{e3.23}, \eqref{e3.10} and
\eqref{e2.1}, we obtain
\begin{align*}
{\rm J_1}&\ls\int_{8
r}^{8\rho(x_0)}\Big(\frac{r}{t}\Big)^\dz\lf[\int_\cx
\frac1{V_t(x)+V(x,\,z)}
\lf(\frac{t}{t+d(x,\,z)}\r)^{\gz}|f(z)-f_B|\,d\mu(z)\r]^2\frac{dt}{t}\\
&\ls\int_{8 r}^{8\rho(x_0)}\Big(\frac{r}{t}\Big)^\dz
\bigg[\frac{1}{\mu(2B)}\int_{2B}|f(z)-f_B|\,d\mu(z)\\
&\hs\hs+\sum_{j=1}^\fz\frac{ t^{\gz}}{(t+2^{j-1}r)^{\gz}}
\frac1{\mu(2^{j+1}B)}\int_{2^{j+1}B}|f(z)-f_B|
\,d\mu(z)\bigg]^2\frac{dt}{t}\\
&\ls\int_{8
r}^{8\rho(x_0)}\Big(\frac{r}{t}\Big)^\dz\bigg[1+\sum_{j=0}^\fz\frac{
t^{\gz}}{(t+2^{j-1}r)^{\gz}}\bigg]^2\frac{dt}{t}.
\end{align*}
Moreover, if $2\gz<\dz$, we then have
$${\rm J_1}\ls\int_{8r}^\fz\Big(\frac{r}{t}\Big)^\dz\frac{dt}{t}+
r^{\dz-2\gz}\int_{8 r}^\fz\frac{dt}{t^{\dz-2\gz+1}}\ls1;$$ if
$2\gz\ge\dz$, letting $\ez\in(0,\,\dz/2)$ yields that
$${\rm J_1}\ls1+\int_{8
r}^{8\rho(x_0)}\Big(\frac{r}{t}\Big)^\dz\bigg[\sum_{j=0}^\fz\frac{
t^{\gz}}{t^{\gz-\ez}(2^{j-1}r)^{\ez}}\bigg]^2\frac{dt}{t}\ls1+
r^{\dz-2\ez}\int_{8 r}^\fz\frac{dt}{t^{\dz-2\ez+1}}\ls1.$$ Thus,
${\rm J_1}\ls1$.

Notice that $r<\rho(x_0)/8$ and $t\in(8 r,\,8\rho(x_0))$. By
\eqref{e3.9}, we have that for any $x\in B$ and $y\in\cx$ with
$d(x,\,y)<t$, $\rho(x_0)\sim\rho(x)\sim\rho(y)$. Choosing
$\eta\in(0,\,1)$ such that $\eta\dz_2<\dz$, then by \eqref{e3.18},
$(Q)_{\rm ii}$ and \eqref{e3.10}, we have
$${\rm J_2}\ls\int_{8
r}^{8\rho(x_0)}\Big[\log\frac{\rho(x_0)}{r}\Big]^2\Big(\frac{r}{t}\Big)^\dz
\Big(\frac{t}{\rho(x_0)}\Big)^{\eta\dz_2}\frac{dt}{t}\ls\int_{8
r}^\fz\Big(\frac{r}{t}\Big)^{\dz-\eta\dz_2}\frac{dt}{t}\ls1.$$
Combining the estimates for ${\rm J_1}$ and ${\rm J_2}$ yields
\eqref{e3.22}, which completes the proof of Theorem \ref{t3.1}.
\end{proof}

As a consequence of Theorem \ref{t3.1}, we have the following
conclusion, which can be proved by an argument similar to the proof
of \cite[Corollary 6.1]{yyz}. We omit the details.
\begin{cor}\label{c3.1}
With the assumptions same as in Theorem \ref{t3.1}, then there
exists a positive constant $C$ such that for all
$f\in\bmo_{\rho}(\cx)$, $S(f)\in\blo_{\rho}(\cx)$ and
$\|S(f)\|_{\blo_{\rho}(\cx)} \le C\|f\|_{\bmo_{\rho}(\cx)}.$
\end{cor}

\begin{rem}\label{r3.1}\rm
In Theorem \ref{t3.1} and Corollary \ref{c3.1}, if we replace the
assumption that the Littlewood-Paley $g$-function in \eqref{e3.2} is
bounded on $L^2(\cx)$ by that the Lusin-area function $S(f)$ in
\eqref{e3.3} is bounded on $L^2(\cx)$, then Theorem \ref{t3.1} and
Corollary \ref{c3.1} still hold.
\end{rem}

Now we study the boundedness of $\gl$ function. In this case, $\cx$
is not necessary to have the $\dz$-annular decay property.

\begin{thm}\label{t3.2}\rm
Let $\cx$ be a doubling metric measure space. Let $\rho$ be an
admissible function on $\cx$ and the $\gl$ function $\gl(f)$ as in
\eqref{e3.4} with $\lz\in(3n,\,\fz)$. Assume that the
Littlewood-Paley $g$-function in \eqref{e3.2} is bounded on
$L^2(\cx)$. Then there exists a positive constant $C$ such that for
all $f\in\bmo_{\rho}(\cx)$, $[\gl(f)]^2\in\blo_{\rho}(\cx)$ and
$\|[\gl(f)]^2\|_{\blo_{\rho}(\cx)}\le C\|f\|^2_{\bmo_{\rho}(\cx)}$.
\end{thm}

\begin{proof}\rm
Again, by the homogeneity of $\|\cdot\|_\bmoz$ and
$\|\cdot\|_\bloz$, we may assume that $f\in\bmo_{\rho}(\cx)$ and
$\|f\|_{\bmo_{\rho}(\cx)}=1.$

Let $B\equiv B(x_0,\,r)$. For any nonnegative integer $k$, let
$$
J(k)\equiv\{(y,\,t)\in\cx\times(0,\,\fz):\,d(y,\,x_0)<2^{k+1}r\,\,{\rm
and}\,\,0<t<2^{k+1}r\}.
$$
For any $f\in\bmo_{\rho}(\cx)$ and $x\in\cx$, write
\begin{align*}\label{gl}
[g^*_\lz(f)(x)]^2&=\iint_{J(0)}\lf(\frac{t}{t+d(x,\,y)}\r)^\lz
|Q_t(f)(y)|^2 \frac {d\mu(y)}{V_t(y)}\frac {dt}{t}
+\iint_{[\cx\times(0,\,\fz)]\setminus
J(0)}\cdots\\
&\equiv [\glz(f)(x)]^2+[\glf(f)(x)]^2.
\end{align*}
We now consider the following two cases. Notice that the
$L^2(\cx)$-boundedness of $g$ via Lemma \ref{l3.1} implies that
$\gl(f)$ is bounded on $L^2(\cx)$.

{\bf Case I.}\quad$r\ge\rho(x_0)$. In this case, we first prove that
\begin{equation}\label{e3.24}
\frac1{\mu(B)}\int_B[\glz(f)(x)]^2\,d\mu(x)\ls1.
\end{equation}

For any $x\in B$, write
\begin{align*}
[\glz(f)(x)]^2&\le\iint_{\gfz{J(0)}{d(x,\,y)<t}}\lf(\frac{t}{t+d(x,\,y)}\r)^\lz
|Q_t(f)(y)|^2 \frac {d\mu(y)}{V_t(y)}\frac {dt}{t}\\
&\hs+\iint_{\gfz{J(0)}{d(x,\,y)\ge
t}}\lf(\frac{t}{t+d(x,\,y)}\r)^\lz
|Q_t(f\chi_{8B})(y)|^2 \frac {d\mu(y)}{V_t(y)}\frac {dt}{t}\\
&\hs+\iint_{\gfz{J(0)}{d(x,\,y)\ge
t}}\lf(\frac{t}{t+d(x,\,y)}\r)^\lz
|Q_t(f\chi_{(8B)^\complement})(y)|^2 \frac {d\mu(y)}{V_t(y)}\frac {dt}{t}\\
&\equiv{\rm I_1}(x)+{\rm I_2}(x)+{\rm I_3}(x).
\end{align*}

Notice that for all $x\in B$, ${\rm I_1}(x)\le[S(f)(x)]^2$. It then
follows from \eqref{e3.5} that
\begin{equation}\label{e3.25}
\frac1{\mu(B)}\int_B{\rm I_1}(x)\,d\mu(x)\ls1.
\end{equation}
We remark that in the proof of \eqref{e3.5}, we do not need the
$\dz$-annular decay property of $\cx$.

As for ${\rm I_2}(x)$, by the $L^2(\cx)$-boundedness of $\gl(f)$,
\eqref{e2.1} and Lemma \ref{l2.1}, we have
\begin{equation}\label{e3.26}
\frac1{\mu(B)}\int_B{\rm I_2}(x)\,d\mu(x)\ls
\frac1{\mu(B)}\int_{8B}|f(x)|^2\,d\mu(x)\ls1.
\end{equation}

To deal with ${\rm I_3}(x)$, we notice that for all
$z\in(8B)^{\complement}$  and $y\in\cx$ with $d(y,\,x_0)< 2r$,
$d(y,\,z)\sim d(x_0,\,z)$ and $V(y,\,z)\sim V(x_0,\,z)$. Hence,
\begin{align*}
{\rm
I_3}(x)&\ls\int_0^{2r}\!\!\!\int_{\gfz{d(y,\,x_0)<2r}{d(x,\,y)\ge
t}}
\Big(\frac{t}{t+d(x,\,y)}\Big)^\lz\\
&\hs\hs\times\bigg[\int_{(8B)^{\complement}}\frac{1}{V_t(y)+V(y,\,z)}
\Big(\frac{t}{t+d(y,\,z)}\Big)^\gz|f(z)|\,d\mu(z)\bigg]^2
\frac {d\mu(y)}{V_t(y)}\frac {dt}{t}\\
&\ls\int_0^{2r}\!\!\!\int_{\gfz{d(y,\,x_0)<2r}{d(x,\,y)\ge t}}
\Big(\frac{t}{t+d(x,\,y)}\Big)^\lz
\bigg[\int_{(8B)^{\complement}}\frac{1}{V(x_0,\,z)}
\Big(\frac{t}{d(x_0,\,z)}\Big)^\gz|f(z)|\,d\mu(z)\bigg]^2
\frac {d\mu(y)}{V_t(y)}\frac {dt}{t}\\
&\ls\int_0^{2r}\Big(\frac{t}{r}\Big)^{2\gz}\sum_{k=0}^\fz
\int_{2^kt\le d(x,\,y)<2^{k+1}t}2^{-k\lz}2^{kn}\frac
{d\mu(y)}{V_{2^{k+1}t}(x)}\frac {dt}{t}\ls1,
\end{align*}
where in the last inequality we used the fact that $\lz>n$.
Furthermore, we obtain
\begin{equation*}
\frac1{\mu(B)}\int_B{\rm I_3}(x)\,d\mu(x)\ls1,
\end{equation*}
which together with \eqref{e3.25} and \eqref{e3.26} proves
\eqref{e3.24}.

Now we prove that
\begin{equation}\label{e3.27}
\frac1{\mu(B)}\int_B[\glf(f)(x)]^2\,d\mu(x)\ls1.
\end{equation}
Notice that for $(y,\,t)\in J(k)\setminus J(k-1)$ with $k\in\nn$ and
$x\in B$, $t+d(x,\,y)\sim 2^k r$. Thus,
\begin{align*}
&[\glf(f)(x)]^2\\
&\hs\ls\sum_{k=1}^\fz\iint_{J(k)\setminus J(k-1)}\Big(\frac{t}{2^k
r}\Big)^\lz\\
&\hs\hs\times\lf[\int_{2^{k+4}B}
\frac{1}{V_t(y)+V(y,\,z)}\Big(\frac{t}{t+d(y,\,z)}\Big)^\gz
\Big(\frac{\rho(y)}{t+\rho(y)}\Big)^{\dz_1}|f(z)|\,d\mu(z)\r]^2
\frac {d\mu(y)}{V_t(y)}\frac {dt}{t}\\
&\hs\hs+\sum_{k=1}^\fz\iint_{J(k)\setminus J(k-1)}\Big(\frac{t}{2^k
r}\Big)^\lz\lf[\int_{(2^{k+4}B)^{\complement}}\cdots\,d\mu(z)\r]^2
\frac {d\mu(y)}{V_t(y)}\frac {dt}{t} \equiv{\rm E_1}(x)+{\rm
E_2}(x).
\end{align*}

The fact that $r\ge\rho(x_0)$ and \eqref{e2.2} imply that for all
$y\in\cx$ with $d(y,\,x_0)<2^{k+1}r$,
\begin{equation}\label{e3.28}
\rho(y)\ls[\rho(x_0)]^{\frac{1}{1+k_0}}(2^{k}r)^{\frac{k_0}{1+k_0}}.
\end{equation}

By the assumption that $\lz\in(3n,\,\fz)$, we choose
$\eta_1\in(0,\,\dz_1)$ such that $\lz-2\eta_1-3n>0$. By
\eqref{e3.28}, we obtain
\begin{align*}
{\rm
E_1}(x)&\ls\sum_{k=1}^\fz\int_0^{2^{k+1}r}\!\!\!\int_{d(y,\,x_0)<2^{k+1}r}
\Big(\frac{t}{2^k r}\Big)^\lz\Big(\frac{2^k r}{t}\Big)^{2n}
\Big(\frac{[\rho(x_0)]^{\frac{1}{1+k_0}}(2^{k}r)^{\frac{k_0}{1+k_0}}}{t}\Big)^{2\eta_1}
\frac {d\mu(y)}{V_t(y)}\frac {dt}{t}\\
&\ls\sum_{k=1}^\fz\int_0^{2^{k+1}r}\Big(\frac{t}{2^k r}\Big)^{\lz}
\Big(\frac{2^k r}{t}\Big)^{3n}
\Big(\frac{[\rho(x_0)]^{\frac{1}{1+k_0}}(2^{k}r)^{\frac{k_0}{1+k_0}}}{t}\Big)^{2\eta_1}
\frac {dt}{t}\ls\sum_{k=1}^\fz\Big[\frac{\rho(x_0)}{2^k
r}\Big]^{\frac{2\eta_1}{1+k_0}}\ls1.
\end{align*}

Choose $\eta_2\in(0,\,\dz_1)$ such that $\lz+2\gz-2\eta_2-n>0$, then
by \eqref{e3.28} and the fact that for $z\in(2^{k+4}B)^\complement$
and $y\in\cx$ with $d(y,\,x_0)< 2^{k+1}r$, $d(y,\,z)\sim d(x_0,\,z)$
and $V(y,\,z)\sim V(x_0,\,z)$, we have
\begin{align*}
{\rm
E_2}(x)&\ls\sum_{k=1}^\fz\int_0^{2^{k+1}r}\!\!\!\int_{d(y,\,x_0)<2^{k+1}r}
\Big(\frac{t}{2^k r}\Big)^\lz\\
&\hs\hs\times\lf[\int_{(2^{k+4}B)^{\complement}}
\frac{1}{V(x_0,\,z)}\Big(\frac{t}{d(x_0,\,z)}\Big)^\gz
\Big(\frac{\rho(y)}{t+\rho(y)}\Big)^{\eta_2}|f(z)|\,d\mu(z)\r]^2
\frac {d\mu(y)}{V_t(y)}\frac {dt}{t}\\
&\ls\sum_{k=1}^\fz\int_0^{2^{k+1}r}\!\!\!\int_{d(y,\,x_0)<2^{k+1}r}
\Big(\frac{t}{2^k r}\Big)^\lz\Big(\frac{t}{2^k r}\Big)^{2\gz}
\Big(\frac{[\rho(x_0)]^{\frac{1}{1+k_0}}(2^{k}r)^{\frac{k_0}{1+k_0}}}{t}\Big)^{2\eta_2}
\frac {d\mu(y)}{V_t(y)}\frac {dt}{t}\\
&\ls\sum_{k=1}^\fz\int_0^{2^{k+1}r}\Big(\frac{t}{2^k
r}\Big)^{\lz+2\gz-n}
\Big(\frac{[\rho(x_0)]^{\frac{1}{1+k_0}}(2^{k}r)^{\frac{k_0}{1+k_0}}}{t}\Big)^{2\eta_2}
\frac {dt}{t} \ls\sum_{k=1}^\fz\Big[\frac{\rho(x_0)}{2^k
r}\Big]^{\frac{2\eta_2}{1+k_0}}\ls1,
\end{align*}
which together with the estimate of ${\rm E_1}(x)$ yields
\eqref{e3.27}.

Combining \eqref{e3.24} and \eqref{e3.27} yields that
\begin{equation}\label{e3.29}
\frac{1}{\mu(B)}\int_B[\gl(f)(x)]^2\,d\mu(x)\ls1.
\end{equation}
Moreover, from the fact that \eqref{e3.29} holds for all balls
$B(x_0,\,r)$ with $r\ge\rho(x_0)$, it follows that $\gl(f)(x)<\fz$
for almost every $x\in\cx$.

{\bf Case II.} \quad $r<\rho(x_0)$. In this case, if
$r\ge\rho(x_0)/16$, then by \eqref{e2.1} and \eqref{e3.29}, we
obtain the desired estimate that
\begin{equation*}
\frac1{\mu(B)}\int_B\lf\{[\gl(f)(x)]^2-{\mathop\einf_B}[\gl(f)]^2\r\}\,d\mu(x)
\ls \frac1{\mu(8 B)}\dint_{8 B}[\gl(f)(x)]^2\,d\mu(x)\ls1.
\end{equation*}

If $r<\rho(x_0)/16$, it is enough to show that for all $x'\in B$
such that $\glf(f)(x')<\fz$,
\begin{eqnarray*}
\frac1{\mu(B)}\int_B\lf\{[\glz(f)(x)]^2+
[\glf(f)(x)]^2-[\glf(f)(x')]^2 \r\}\,d\mu(x)\ls1.
\end{eqnarray*}

We first prove that
\begin{equation}\label{e3.30}
\frac1{\mu(B)}\int_B[\glz(f)(x)]^2\,d\mu(x)\ls1.
\end{equation}
To this end, write $f\equiv f_1+f_2+f_B,$ where
$f_1\equiv(f-f_B)\chi_{8B}$ and
$f_2\equiv(f-f_B)\chi_{(8B)^\complement}$. By the
$L^2(\cx)$-boundedness of $\gl(f)$, \eqref{e2.1} and Lemma
\ref{l2.1}, we have
\begin{equation}\label{e3.31}
\frac1{\mu(B)}\int_B[\glz(f_1)(x)]^2\,d\mu(x) \ls
\frac1{\mu(B)}\int_{8B}|f-f_B|^2\,d\mu(x)\ls1.
\end{equation}
Notice that for $z\in(8B)^\complement$ and $y\in\cx$ with
$d(y,\,x_0)< 2r$, $d(y,\,z)\sim d(x_0,\,z)$ and $V(y,\,z)\sim
V(x_0,\,z)$. This together with ${\rm (Q)_{i}}$, \eqref{e2.1} and
the fact that $|f_{2^{j+1}B}-f_B|\ls j$ for all $j\in\nn$ yields
that
\begin{align*}
|Q_t(f_2)(y)|&\ls\int_{(8B)^\complement}
\frac1{V_t(y)+V(y,\,z)}\lf(\frac{t}{t+d(y,\,z)}\r)^{\gz}
|f(z)-f_B|\,d\mu(z)\\
&\ls\int_{(8B)^\complement}
\frac1{V(x_0,\,z)}\lf(\frac{t}{d(x_0,\,z)}\r)^{\gz}
|f(z)-f_B|\,d\mu(z)\ls\lf(\frac tr\r)^{\gz},
\end{align*}
where we omitted some routine computation. Hence, by an argument
similar to the estimates of \eqref{e3.11} and ${\rm I_3}(x)$, we
obtain
\begin{align}\label{e3.32}
&\frac1{\mu(B)}\int_B[\glz(f_2)(x)]^2\,d\mu(x)\\
&\hs\ls\frac1{\mu(B)}\int_B\int_0^{2r}\int_{d(x_0,\,y)<2r}
\bigg(\frac{t}{t+d(x,\,y)}\bigg)^\lz\lf(\frac tr\r)^{2\gz}
\,\frac{d\mu(y)}{V_t(y)}\frac{dt}t\,d\mu(x)\nonumber\\
&\hs\ls\frac1{\mu(B)}\int_B\int_0^{2r}\int_{d(x,\,y)<t} \lf(\frac
tr\r)^{2\gz}
\,\frac{d\mu(y)}{V_t(y)}\frac{dt}t\,d\mu(x)\nonumber\\
&\hs\hs+\frac1{\mu(B)}\int_B\int_0^{2r}
\int_{\gfz{d(x_0,\,y)<2r}{d(x,\,y)\ge t}}
\bigg(\frac{t}{t+d(x,\,y)}\bigg)^\lz\lf(\frac tr\r)^{2\gz}
\,\frac{d\mu(y)}{V_t(y)}\frac{dt}t\,d\mu(x)\ls1.\nonumber
\end{align}
For $y\in\cx$ with $d(x_0,\,y)<2r<\rho(x_0)/8$, by \eqref{e3.9}, we
have $\rho(x_0)\sim\rho(y)$, which together with $(Q)_{\rm ii}$ and
\eqref{e3.18} leads to
$$
|Q_t(f_B)(y)|\ls \lf(\frac t{\rho(y)}\r)^{\dz_2}|f_B|\ls \lf(\frac
t{\rho(x_0)}\r)^{\dz_2}
\log\frac{\rho(x_0)}{r}\ls\Big(\frac{t}{r}\Big)^{\dz_2}.
$$ Then, similarly to
the estimate of \eqref{e3.32}, we obtain
\begin{equation*}
\frac{1}{\mu(B)}\int_B[\glz(f_B)(x)]^2\,d\mu(x)\ls1,
\end{equation*}
which together with \eqref{e3.31} and \eqref{e3.32} yields
\eqref{e3.30}.

The proof of Theorem \ref{t3.2} now is reduced to show that for all
$x'\in B$ such that $\glf(f)(x')<\fz$,
\begin{equation}\label{e3.33}
\frac1{\mu(B)}\int_B\lf\{[\glf(f)(x)]^2-[\glf(f)(x')]^2
\r\}\,d\mu(x)\ls1.
\end{equation}
For $x,\,x'\in B$ such that $\glf(x)$ and $\glf(x')$ are finite,
write
\begin{align*}
&[\glf(f)(x)]^2-[\glf(f)(x')]^2\\
&\hs\le\iint_{\cx\times(0,\,\fz)\setminus
J(0)}\lf|\lf(\frac{t}{t+d(x,\,y)}\r)^\lz
-\lf(\frac{t}{t+d(x',\,y)}\r)^\lz\r|
|Q_t(f)(y)|^2 \frac {d\mu(y)}{V_t(y)}\frac {dt}{t}\\
&\hs\ls\sum_{k=1}^\fz\iint_{J(k)\setminus J(k-1)}\frac{rt^\lz}{(2^k
r)^{\lz+1}}|Q_t(f-f_B)(y)|^2\frac{d\mu(y)}{V_t(y)}\frac{dt}{t}\\
&\hs\hs+\sum_{k=1}^\fz\iint_{J(k)\setminus J(k-1)}\frac{rt^\lz}{(2^k
r)^{\lz+1}}|Q_t(f_B)(y)|^2\frac{d\mu(y)}{V_t(y)}\frac{dt}{t}\equiv{\rm
G_1}+{\rm G_2}.
\end{align*}
Using the assumption that $\lz\in(3n,\,\fz)$ and $(Q)_{\rm i}$, we
have
\begin{align*}
{\rm G_1} &\ls\sum_{k=1}^\fz\iint_{J(k)\setminus
J(k-1)}\frac{rt^\lz}{(2^k
r)^{\lz+1}}\\
&\hs\times\lf[\int_{2^{k+4}B}
\frac{1}{V_t(y)+V(y,\,z)}\Big(\frac{t}{t+d(y,\,z)}\Big)^\gz
|f(z)-f_B|\,d\mu(z)\r]^2
\frac {d\mu(y)}{V_t(y)}\frac {dt}{t}\\
&\hs+\sum_{k=1}^\fz\iint_{J(k)\setminus J(k-1)}\frac{rt^\lz}{(2^k
r)^{\lz+1}}\lf[\int_{(2^{k+4}B)^{\complement}} \cdots\,d\mu(z)\r]^2
\frac {d\mu(y)}{V_t(y)}\frac {dt}{t}\\
&\ls\sum_{k=1}^\fz\int_0^{2^{k+1}r}\!\!\!\int_{d(y,\,x_0)<2^{k+1}r}
\frac{rt^\lz}{(2^k r)^{\lz+1}}\Big(\frac{2^k r}{t}\Big)^{2n}k^2
\frac {d\mu(y)}{V_t(y)}\frac
{dt}{t}\\
&\hs+\sum_{k=1}^\fz\int_0^{2^{k+1}r}\!\!\!\int_{d(y,\,x_0)<2^{k+1}r}
\frac{rt^\lz}{(2^k r)^{\lz+1}}\Big(\frac{t}{2^k r}\Big)^{2\gz}
k^2\frac {d\mu(y)}{V_t(y)}\frac {dt}{t}\ls1.
\end{align*}
Choose $\eta_3\in(0,\,1)$ such that $\eta_3(1+k_0)\dz_2<1$. It then
follows from $(Q)_{\rm ii}$, \eqref{e2.2}, \eqref{e3.18} and
$\lz\in(n,\,\fz)$ that
\begin{align*}
{\rm
G_2}&\ls\sum_{k=1}^\fz\int_0^{2^{k+1}r}\!\!\!\int_{d(y,\,x_0)<2^{k+1}r}
\frac{rt^\lz}{(2^k r)^{\lz+1}}\Big[\log\frac{\rho(x_0)}{r}\Big]^2
\Big(\frac{t}{\rho(y)}\Big)^{\eta_3\dz_2}\,\frac{d\mu(y)}{V_t(y)}\frac{dt}{t}\\
&\ls\sum_{k=1}^\fz\int_0^{2^{k+1}r}\frac{rt^\lz}{(2^k
r)^{\lz+1}}\Big[\log\frac{\rho(x_0)}{r}\Big]^2 \bigg[\Big(\frac{2^k
r}{\rho(x_0)}\Big)^{\eta_3\dz_2} +\Big(\frac{2^k
r}{\rho(x_0)}\Big)^{\eta_3(1+k_0)\dz_2}\bigg]\Big(\frac{2^k
r}{t}\Big)^n\frac{dt}{t}\ls1.
\end{align*}
Combining the estimates for ${\rm G_1}$ and ${\rm G_2}$ yields
\eqref{e3.33}, which completes the proof of Theorem \ref{t3.2}.
\end{proof}

As a consequence of Theorem \ref{t3.2}, we have the following
conclusion.
\begin{cor}\label{c3.2}
With the assumptions same as in Theorem \ref{t3.2}, then there
exists a positive constant $C$ such that for all
$f\in\bmo_{\rho}(\cx)$, $\gl(f)\in\blo_{\rho}(\cx)$ and
$\|\gl(f)\|_{\blo_{\rho}(\cx)} \le C\|f\|_{\bmo_{\rho}(\cx)}.$
\end{cor}

\begin{rem}\label{r3.2}\rm
(i) In Theorem \ref{t3.2} and Corollary \ref{c3.2}, if we replace
the assumption that the Littlewood-Paley $g$-function in
\eqref{e3.2} is bounded on $L^2(\cx)$ by that the $\gl$ function
$\gl(f)$ in \eqref{e3.4} is bounded on $L^2(\cx)$, then Theorem
\ref{t3.2} and Corollary \ref{c3.2} still hold.

(ii) Comparing with the classical known result in \cite{my08}, it is
still unclear if $\lz\in(n,\,\fz)$ is enough to guarantee Theorem
\ref{t3.2} and Corollary \ref{c3.2}. In the proof of Theorem
\ref{t3.2}, we need the assumption $\lz>3n$ only in the estimates of
${\rm E_1}(x)$ and ${\rm G_1}$. In \cite{my08}, this can be reduced
to $\lz>n$ via the fractional integral. However, in the current
setting, corresponding result of the fractional integral is not
available.

(iii) Let $\cx=(\rd,\,|\cdot|,\,dx)$ and $\{Q_t\}_{t>0}$ be the
operators associated to the semigroups generated by the
Schr\"odinger operator with nonnegative potential satisfying the
reverse H\"older inequality on $\rd$; see Proposition \ref{p3.1}
below. Then, Theorem \ref{t3.2} implies that the $\gl$ function
$\gl(f)$ associate to the kernels $\{Q_t\}_{t>0}$ is bounded from
$\mathrm{BMO}_\rho(\rd)$ to $\mathrm{BLO}_\rho(\rd)$ for
$\lz\in(3d,\,\fz)$, which improves the result in \cite{hl2} that
$\gl(f)$ is bounded on $\mathrm{BMO}_\rho(\rd)$ for
$\lz\in(3d+4k_0,\,\fz)$, where $k_0$ is as in \eqref{e2.2}.
\end{rem}

Notice that Buckley \cite{b99} showed that Heisenberg groups and
connected and simply connected nilpotent Lie groups with a
Carnot-Carath\'eodory (control) distance have the $\dz$-annular
decay property (see also Example \ref{example4.1}
below). By this fact, we have the following simple corollary
of Theorems \ref{t3.1} and \ref{t3.2}, and Corollaries \ref{c3.1}
and \ref{c3.2}. We omit the details here; see \cite[Section 7]{yyz}.

\begin{prop}\label{p3.1} Theorems \ref{t3.1} and \ref{t3.2}, and
Corollaries \ref{c3.1} and \ref{c3.2} are true if
$$
Q_t\equiv t^2\frac{de^{-s\mathcal {L}}}{ds}\bigg|_{s=t^2},
$$
where $\mathcal{L}=-\Delta+V$ is the Schr\"odinger operator or the
degenerate Schr\"odinger operator on ${{\mathbb R}}^d$, or the
sub-Laplace Schr\"odinger operator on Heisenberg groups or connected
and simply connected nilpotent Lie groups, and $V$ is a nonnegative
function satisfying certain reverse H\"older inequality; see the
details in \cite[Section 7]{yyz}.
\end{prop}

\section{Several remarks on the $\dz$-annular decay property\label{s4}}

\hskip\parindent To the best of our knowledge, the $\dz$-annular
decay property in Definition \ref{d3.1} was introduced by Buckley
\cite{b99} in 1999. However, if $(\cx, d, \mu)$ is a normal space of
homogeneous type in the sense of Marc\'ias and Segovia \cite{ms79},
the $\dz$-annular decay property was introduced by David, Journ\'e
and Semmes in 1985 in their celebrated paper on the $T(b)$ theorem
(see \cite[p.\,41]{djs85}). A slight variant on manifolds also
appeared in Colding and Minicozzi II \cite{cm98} in 1998, which was
called $\ez$-volume regularity property therein (see
\cite[p.\,125]{cm98}).
Buckley \cite{b99} proved that for any metric space equipped with a
doubling measure, the chain ball property implies the
$\delta$-annular decay property for some $\dz\in (0,1]$.

In this section, we first introduce two properties on any metric space, the
weak geodesic property and the monotone geodesic property, which are
proved to be respectively equivalent to the chain ball property
introduced by Buckley \cite{b99}. As an application, we prove that
any length space equipped with a doubling measure has the weak geodesic
property and hence the $\dz$-annular decay property for some $\dz\in
(0,1]$. Finally, we give several examples of doubling
metric measure spaces having the $\dz$-annular decay property.

We begin with the notions of the weak geodesic property,
the monotone geodesic property, and the chain ball property.

\begin{defn}\label{d4.1}\rm
Let $(\cx,d)$ be a metric space.

(I) $(\cx,\, d)$ is said to have the {\it weak geodesic property}
(or called {\it Property $(\wz{M})$}) if there exists a positive constant $C_3$ such that for
all $x\in\cx$, $r,\,s\in(0,\,\fz)$ and $y\in \ol{B(x,\,r+s)}$,
$d(y,\,\ol{B(x,\,r)})\le C_3s$.

(II) $(\cx,\, d)$ is said to have {\it the monotone geodesic
property} if there exists a positive constant $C_4$ such that for
all $s>0$ and $x,\,y\in\cx$ with $d(x,\,y)\geq s$, there exists a
finite chain $x_0\equiv y,\,x_1,\,\cdots,\,x_m\equiv x$ with
$m\in\nn$ such that for $0\le i<m$, $d(x_i,\,x_{i+1})\le C_4s$ and
$d(x_{i+1},\,x)\le d(x_i,\,x)-s.$

(III) Let $\az,\bz\in(1,\,\fz)$. A ball $B\equiv B(z,r)\subset \cx$ is
said to be an {\it $(\az,\bz)$-chain ball}, with respect to a
``central" sub-ball $B_0\equiv B(z_0,r_0)\subset B$ if, for every $x\in B$,
there is an integer $k\equiv k(x)\ge0$ and a chain of balls,
$B_{x,i}\equiv B(z_{x,i}, r_{x,i})$, $0\le i\le k$, with the following
properties:

(i) $B_{x,0}=B_0$ and $x\in B_{x,k}$,

(ii) $B_{x,i}\cap B_{x,i+1}$ is non-empty, $0\le i<k$,

(iii) $x\in \az B_{x,i}$, $0\le i\le k$,

(iv) $\bz r_{x,i}\le r-d(z_{x,i},z)$, $0\le i\le k$. \\
The metric space $(\cx,\,d)$ is said to have the {\it $(\az,\bz)$-chain ball property}
if every ball in $\cx$ is an $(\az,\bz)$-chain ball.
\end{defn}

\begin{rem}\label{r4.1}\rm
(i) Tessera in \cite{te} introduced the following Property (M). A
metric space $(\cx,\,d)$ is said to has {\it Property (M)} if there
exists a positive constant $C$ such that the Hausdorff distance
between any pair of balls with same center and any radii between $r$
and $r+1$ is less than $C$. In other words, there exists a positive
constant $C$ such that for all $x\in\cx$, $r>0$ and
$y\in\ol{B(x,\,r+1)}$, $d(y,\,\ol{B(x,\,r)})\le C$; see
\cite[Definition 1]{te}. Obviously, if $(\cx,\,d)$ has Property
$(\wz M)$, then $(\cx,\,d)$ also has Property $(M)$.

Conversely, let $\zz$ be equipped with the usual Euclidean distance
$|\cdot|$. Then $(\zz,\,|\cdot|)$ has Property $(M)$. Assume that
$(\zz,\,|\cdot|)$ has also Property $(\wz M)$. Then, by Definition
\ref{d4.1}(I), there exists a positive constant $C_3$ such that for
all $r,\,s\in(0,\,\fz)$ and $y\in \ol{B(0,\,r+s)}$,
$d(y,\,\ol{B(0,\,r)})\le C_3s$. If we choose $s<\min\{1,
(C_3)^{-1}\}$ and $r\in(0,1)$ with $r+s\ge 1$, then $C_3s<1$,
$\ol{B(0,\,r)}=\{0\}$ and $\ol{B(0,\,r+s)}=\{0,\,1\}$, it then
follows that $1=d(1,\,\ol{B(0,\,r)})\le C_3s<1$, which is a
contradiction. Thus, $(\zz,\,|\cdot|)$ does not have Property $(\wz
M)$. In this sense, we say that Property $(\wz M)$ is slightly
stronger than Property $(M)$.

(ii) Let $(\cx,\,\mu,\,d)$ be a doubling measure space having
Property $(M)$. Then using 3) of Proposition 2 in \cite{te}, by an
argument same as in the proof of Theorem 4 of \cite{te} (see also
the proof of Lemma 3.3 of Colding and Minicozzi II \cite{cm98}), we
have that there exist positive constants $\dz$ and $C$ such that for
all $x\in\cx$, $s\in[1,\fz)$ and $r\in(s,\,\fz)$,
\begin{align*}
\mu(B(x,\,r+s))-\mu(B(x,\,r))\le
C\lf(\frac{s}{r}\r)^\dz\mu(B(x,\,r)).
\end{align*}
Thus, when $\dz\in(0,\,1]$, $(\cx,\,\mu,\,d)$ satisfies a slightly
weaker property than the $\dz$-annular decay property.

Tessera in \cite[pp.\,51-52]{te} also verified that the assumptions
of Theorem 4 in \cite{te} are optimal. Thus, in some sense, it is
necessary to introduce the weak geodesic property to guarantee the
$\dz$-annular decay property.

(iii) It is easy to check that $C_4\ge1$. In fact, if $m=1$, that
is, $x_0\equiv y$ and $x_1\equiv x$, then $s\le
d(x,\,y)=d(x_0,\,x_1)\le C_4s$, which implies that $C_4\ge1$; if
$m>1$, that is, $x_0\equiv y,\,x_1,\,\cdots,\,x_m\equiv x$, then
$d(y,\,x_1)=d(x_0,\,x_1)\le C_4s$ and $d(x_1,\,x)\le
d(x_0,\,x)-s=d(y,\,x)-s$, which also implies that
$s=d(x,\,y)-(d(x,\,y)-s)\le d(x,\,y)-d(x_1,\,x)\le d(y,\,x_1)\le
C_4s$ and hence $C_4\ge1$.

(iv) The notion of $(\az,\bz)$-chain ball property in Definition
\ref{d4.1}(III) was first introduced by Buckley in \cite{b99}.
Moreover, it is easy to see that in Definition \ref{d4.1}(III),
$B_{x,\,i}\subset B$ for all $x\in B$ and $i\in\{0,\,\cdots,\,k\}$.
In fact, by (iv) of Definition \ref{d4.1}(III) and the fact that
$\bz\in(1,\,\fz)$, we have that for any $w\in B_{x,\,i}$,
$d(w,\,z)\le d(w,\, z_{x,\,i})+d(z_{x,\,i},\,z)<
r_{x,\,i}+d(z_{x,\,i},\,z)<\bz r_{x,\,i}+d(z_{x,\,i},\,z)\le r$,
which implies that $B_{x,\,i}\subset B$.
\end{rem}

The main result of this section is the following equivalences of the
above three properties.

\begin{thm}\label{t4.1}
Let $(\cx,\,d)$ be a metric space. Then the following are
equivalent:

{\rm (I)} $(\cx,\,d)$ has the weak geodesic property;

{\rm (II)} $(\cx,\,d)$ has the monotone geodesic property;

{\rm (III)} $(\cx,\,d)$ has the $(\az,\,\bz)$-chain ball property
for some $\az,\,\bz\in(1,\,\fz)$.
\end{thm}

\begin{proof}
Similarly to the proof of \cite[Proposition 2]{te}, we can show the
equivalence of (I) and (II). We omit the details.

Now we prove that (II) implies (III). To this end, let $(\cx,\,d)$
be a metric space having the monotone geodesic property with a
positive constant $C_4$, and let $B\equiv B(z,r)$ be any ball in
$\cx$. We show that $B$ is a $(4C_4/3,4/3)$-chain ball with respect
to the``central" sub-ball $B_0\equiv B(z,3r/4)\subset B$.

For every $x\in B$, let $t_0\equiv(r-d(x,z))/2$. If $x\in B_0$, then
$k\equiv k(x)\equiv 0$ and $\{B_0\}$ is a desired chain.

Assume that $x\notin B_0$, then  $d(x,\,z)\ge 3r/4$ and $t_0\le
r/8$. Thus, $d(x,z)\ge 6t_0>t_0/C_4$, since $C_4\ge 1$ by Remark
\ref{r4.1}(iii). Since $\cx$ has the monotone geodesic property, by
Definition \ref{d4.1}(II), there exists a finite chain
$x_{0,0}\equiv x,\,x_{0,1},\,\cdots,\,x_{0,m_0}\equiv z$ with
$m_0\in\nn$ such that for $0\le i<m_0$, $d(x_{0,i},\,x_{0,i+1})\le
t_0$ and $d(x_{0,i+1},\,z)\le d(x_{0,i},\,z)-t_0/C_4.$ In this case,
$B(x_{0,0},3t_0/2)=B(x,3t_0/2)\ni x_{0,1}$ and
$(4/3)\times3t_0/2=r-d(x,z).$ If $x_{0,1}\in B_0$, then $k\equiv
k(x)\equiv 1$ and $\{B_0,B(x,3t_0/2)\}$ is a desired chain, since
$x\in(4C_4/3)B_0$.

Assume that $x_{0,1}\notin B_0$ and let
$t_1\equiv(r-d(x_{0,1},z))/2$, then  $d(x_{0,\,1},\,z)\ge 3r/4$ and
$t_1\le r/8$. Thus, $d(x_{0,1},z)\ge 6t_1> t_1/C_4$, by $C_4\ge1$.
By Definition \ref{d4.1}(II), there exists a finite chain
$x_{1,0}\equiv x_{0,1},\,x_{1,1},\,\cdots,\,x_{1,m_1}\equiv z$ with
$m_1\in\nn$ such that for $0\le i<m_1$, $d(x_{1,i},\,x_{1,i+1})\le
t_1$ and $d(x_{1,i+1},\,z)\le d(x_{1,i},\,z)-t_1/C_4.$ In this case,
$B(x_{1,0},3t_1/2)=B(x_{0,1},3t_1/2)\ni x_{1,1}$ and
$(4/3)\times3t_1/2=r-d(x_{0,1},z).$ Moreover, $t_0\le
t_1(2C_4)/(1+2C_4)$, since
$$
  t_1-t_0=(r-d(x_{0,1},z))/2-(r-d(x,z))/2=(d(x,z)-d(x_{0,1},z))/2\ge t_0/(2C_4).
$$
Then, $d(x,x_{0,1})\le t_0\le
t_1(2C_4)/(1+2C_4)<2C_4t_1=(4C_4/3)\times(3t_1/2)$, that is,
$$
  x\in (4C_4/3)B(x_{0,1},3t_1/2).
$$
If $x_{1,1}\in B_0$, then $k\equiv k(x)\equiv 2$ and
$\{B_0,B(x_{0,1},3t_1/2),B(x,3t_0/2)\}$ is a desired chain.

Assume that $x_{j,1}\notin B_0$ and let
$t_{j+1}\equiv(r-d(x_{j,1},z))/2$, then $d(x_{j,\,1},\,z)\ge 3r/4$
and $t_{j+1}\le r/8$. Thus, $d(x_{j,1},z)\ge 6t_{j+1}>t_{j+1}/C_4$,
by $C_4\ge 1$. By Definition \ref{d4.1}(II), there exists a finite
chain $x_{j+1,0}\equiv
x_{j,1},\,x_{j+1,1},\,\cdots,\,x_{j+1,m_{j+1}}\equiv z$ with
$m_{j+1}\in\nn$ such that for $0\le i<m_{j+1}$,
$d(x_{j+1,i},\,x_{j+1,i+1})\le t_{j+1}$ and $d(x_{j+1,i+1},\,z)\le
d(x_{j+1,i},\,z)-t_{j+1}/C_4.$ In this case,
$$
B(x_{j+1,0},3t_{j+1}/2)=B(x_{j,1},3t_{j+1}/2)\ni x_{j+1,1}\quad{\rm
and}\quad (4/3)\times3t_{j+1}/2=r-d(x_{j,1},z).
$$
Moreover, $t_j\le
t_{j+1}(2C_4)/(1+2C_4)$, since $x_{j-1,1}=x_{j,0}$ and
$$
t_{j+1}-t_j=(r-d(x_{j,1},z))/2-(r-d(x_{j-1,1},z))/2
=(d(x_{j,0},z)-d(x_{j,1},z))/2\ge t_j/(2C_4).
$$
Then,
\begin{align*}
  d(x,x_{j,1})
  &\le d(x,x_{0,1})+\sum_{\ell=1}^{j}d(x_{\ell-1,1},x_{\ell,1}) \\
  &= d(x_{0,0},x_{0,1})+\sum_{\ell=1}^{j}d(x_{\ell,0},x_{\ell,1}) \\
  &\le \sum_{\ell=0}^{j}t_{\ell}
  \le \sum_{\ell=1}^{j+1}t_{j+1}((2C_4)/(1+2C_4))^{\ell} \\
  &<2C_4t_{j+1}=(4C_4/3)\times(3t_{j+1}/2),
\end{align*}
that is, $x\in (4C_4/3)B(x_{j,1},3t_{j+1}/2).$ If $x_{j+1,1}\in
B_0$, then $k\equiv k(x)\equiv j+2$ and
$$
  \{B_0,B(x_{j,1},3t_{j+1}/2),\cdots,B(x_{0,1},3t_1/2),B(x,3t_0/2)\}
$$
is a desired chain.

To finish the proof that (II) implies (III), we must show
$x_{j_0,1}\in B_0$ for some $j_0\in\nn\cup\{0\}$. To this end, it is
enough to show that
\begin{equation} \label{e4.1}
  t_j\ge\frac12\left(1+\frac1{2C_4}\right)^j[r-d(x,z)]
 \quad{\rm and}\quad
  d(x_{j,1},z)\le r-\left(1+\frac1{2C_4}\right)^{j+1}[r-d(x,z)],
\end{equation}
by induction. By the definitions of $t_0$ and $x_{0,1}$, we have
that $t_0=\frac12(r-d(x,z))$ and
\begin{align*}
  d(x_{0,1},z)&\le d(x_{0,0},z)-t_{0}/C_4=d(x,z)-t_{0}/C_4 \\
  &=d(x,z)-\frac1{2C_4}[r-d(x,z)] \\
  &=r-\left(1+\frac1{2C_4}\right)[r-d(x,z)].
\end{align*}
Then \eqref{e4.1} holds for $j=0$. Assume that \eqref{e4.1} holds
for $j\in\nn$ and we consider the case $j+1$. By the definitions of
$x_{j,1}$ and $t_j$, we have
\begin{align*}
  t_{j+1}
  &=\frac12(r-d(x_{j,1},z)) \\
  &\ge \frac12\left(r-\left\{r-\left(1+\frac1{2C_4}\right)^{j+1}[r-d(x,z)]\right\}\right), \\
  &=\frac12\left(1+\frac1{2C_4}\right)^{j+1}[r-d(x,z)]
\end{align*}
and
\begin{align*}
  d(x_{j+1,1},z)
  &\le d(x_{j+1,0},z)-t_{j+1}/C_4=d(x_{j,1},z)-t_{j+1}/C_4\\
  &\le r-\left(1+\frac1{2C_4}\right)^{j+1}[r-d(x,z)]
    - \frac1{2C_4}\left(1+\frac1{2C_4}\right)^{j+1}[r-d(x,z)] \\
  &=r-\left(1+\frac1{2C_4}\right)^{j+2}[r-d(x,z)].
\end{align*}
Thus, \eqref{e4.1} holds and (II) implies (III).

Finally, we prove that (III) implies (I). Assume that $(\cx,\,d)$
has the $(\az,\,\bz)$-chain ball property for some
$\az,\,\bz\in(1,\,\fz)$, but not the weak geodesic property, that
is, for all natural numbers $N$, there exist $x_N\in\cx$, $r_N,
s_N\in(0,\,\fz)$ and $y_N\in \ol{B(x_N,\,r_N+s_N)}$ such that
$d(y_N,\,\ol{B(x_N,\,r_N)})>Ns_N$. In this case, $y_N\notin
B(x_N,\,r_N)$ and $Ns_N<d(y_N,x_N)\le r_N+s_N$. Thus,
\begin{equation} \label{e4.2}
  (N-1)s_N<r_N.
\end{equation}

We show that, for all $\az,\bz\in(1,\,\fz)$, there exists $N\in\nn$
such that $B(x_N,\,r_N+2s_N)$ is not an $(\az,\bz)$-chain ball.
Otherwise, for some $\az,\,\bz\in(1,\,\fz)$ and for all $N\in\nn$, if
$B(x_N,\,r_N+2s_N)$ is an $(\az,\bz)$-chain ball with respect to
$B_{N,0}\equiv B(z_{N,0},t_{N,0})\subset B(x_N,\,r_N+2s_N)$, then
there exists an integer $k\equiv k(y_N)>0$ and a chain of balls,
$B_{N,i}\equiv B(z_{N,i},t_{N,i})\subset B(x_N,\,r_N+2s_N)$, $0\le
i\le k$, satisfy that

(i) $y_N\in B_{N,k}$,

(ii) $B_{N,i}\cap B_{N,i+1}$ is non-empty, $0\le i<k$,

(iii) $y_N\in B(z_{N,i},\az t_{N,i})$, $0\le i\le k$,

(iv) $\bz t_{N,i}\le r_N+2s_N-d(z_{N,i},x_N)$, $0\le i\le k$. \\
If $N$ satisfies $\bz-{4\az}/{(N-1)}>1$, then
\begin{equation} \label{e4.3}
  \bigcup_{0\le i\le k}B_{N,i}\subset B(x_N,\,r_N),
\end{equation}
that is, $y_N\notin B_{N,i}$ for all $0\le i\le k$, which is
contradicts to (i).

In the following we show \eqref{e4.3}. By (iii) of Definition
\ref{d4.1}(III), we have that $x_N\in B(z_{N,0},\az t_{N,0})$, which
together with  $y_N\notin B(x_N,\,r_N)$, \eqref{e4.2} and (iii)
leads to that
$$
  (N-1)s_N\le r_N\le d(y_N,x_N)\le d(y_N,z_{N,0})+d(z_{N,0},x_N)<2\az t_{N,0}.
$$
Hence, by (iv),
$$
  d(z_{N,0},x_N)\le r_N+2s_N-\bz t_{N,0}
  \le r_N+\left(\frac{4\az}{N-1}-\bz\right)t_{N,0}
  <r_N-t_{N,0},
$$
since $\bz-4\az/(N-1)>1$. Thus, $B(z_{N,0},t_{N,0})\subset
B(x_N,r_N)$, since, for $w\in B(z_{N,0},t_{N,0})$,
$$
  d(w,x_N)\le d(w,z_{N,0})+d(z_{N,0},x_N)<t_{N,0}+r_N-t_{N,0}=r_N.
$$
Assume that $B_{N,i}\subset B(x_N,r_N)$ for $i\in\nn$. We show
$B_{N,i+1}\subset B(x_N,r_N)$. From (ii) and (iii), it follows that
there exists $w\in \lf(B_{N,i}\cap B_{N,i+1}\r)\subset B(x_N,r_N)$
and
$$
  Ns_N<d(w,y_N)\le d(w,z_{N,i+1})+d(z_{N,i+1},y_N)\le (1+\alpha)t_{N,i+1}.
$$
Hence, by (iv),
$$
  d(z_{N,i+1},x_N)\le r_N+2s_N-\bz t_{N,i+1}
  \le r_N+\left(\frac{2(1+\az)}{N}-\bz\right)t_{N,i+1}
  <r_N-t_{N,i+1},
$$
since $\bz-2(1+\az)/N>\bz-4\az/(N-1)>1$. Thus,
$B(z_{N,i+1},t_{N,i+1})\subset B(x_N,r_N)$, which completes the
proof of \eqref{e4.3} and hence Theorem 4.1.
\end{proof}

As an application of the chain ball property, Buckley in \cite{b99}
proved the following useful result.

\begin{lem}[\cite{b99}]\label{l4.1}
Let $\cx=(\cx,\,d,\,\mu)$ be a doubling metric measure space with
doubling constat $C_1$. Suppose that $(\cx,\,d)$ also has the
$(\az,\bz)$-chain ball property for some $\az,\,\bz\in(1,\,\fz)$,
then $\mu$ has the $\dz$-annular decay property for some
$\dz\in(0,\,1]$ dependent only on $\az, \bz$ and $C_1$.
\end{lem}

As a consequence of Theorem \ref{t4.1} and Lemma \ref{l4.1}, we have
the following conclusion.

\begin{cor}\label{c4.1}
Let $\cx=(\cx,\,d,\,\mu)$ be a doubling metric measure space. If
$(\cx,\,d)$ has either the weak geodesic property or the monotone geodesic
property, then $\mu$ has the $\dz$-annular decay
property for some $\dz\in(0,\,1]$.
\end{cor}

\begin{rem}\label{r4.2}\rm
By an argument similar to that used in the proof of \cite[Theorem
4]{te}, we can also directly prove Corollary \ref{c4.1}, without invoking
Lemma \ref{l4.1}. We omit the details.
\end{rem}

As an application of Corollary \ref{c4.1}, we show that any length
space equipped with a doubling measure has the $\dz$-annular decay
property, which is just \cite[Corollary 2.2]{b99}. However,
unlike the proof of \cite[Corollary 2.2]{b99}, we prove the
following Proposition \ref{p4.1} without invoking the property of John
domains. In what follows, for any rectifiable path $\gz$,
let $\ell(\gz)$ denote its length.

\begin{prop}\label{p4.1} Any length space $(\cx,\,d)$ has the weak
geodesic property. Moreover, if $\mu$ is a doubling measure
on $(\cx,\,d)$ with doubling constant $C_1$, then $\mu$ has the
$\dz$-annular decay property for some $\dz\in(0,\,1]$ dependent only
on $C_1$.
\end{prop}

\begin{proof}
Let $x\in\cx$, $r,\,s\in(0,\,\fz)$ and $y\in \ol{B(x,\,r+s)}$. If
$d(x,\,y)\le r$, then $d(y, \ol{B(x,\,r)})=0\le s$.
If $r<d(x,\,y)\le r+s$, then for any given $\ez>0$,
there exists a rectifiable path $\gz$ from $x$ to $y$ such that
$\ell(\gz)<d(x,\,y)+\ez$. Moreover, by the mean value theorem
for the continuous function of $w\mapsto d(x,\,w)$
restricted to the path $\gz$, there exists a $z\in\gz$ such that
$d(x,\,z)=r$. By splitting the path
$\gz$ into $\gz_1$ from $x$ to $z$ and $\gz_2$ from $z$ to $y$, we
have by definition of the distance and choice of $\gz$ that
$d(x,\,z)+d(z,\,y)\le\ell(\gz_1)+\ell(\gz_2)=\ell(\gz)<d(x,\,y)+\ez$.
Thus, $d(y,\ol{B(x,\,r)})\le d(y,\,z)<d(x,\,y)+\ez-d(x,\,z)\le s+\ez$.
Letting $\ez\to0$ yields that $d(y,\,\ol{B(x,\,r)})\le s$, which shows that
$(\cx,\,d)$ has the weak
geodesic property. This combined with Corollary \ref{c4.1}
implies that $\mu$ has the
$\dz$-annular decay property for some $\dz\in(0,\,1]$ dependent only
on $C_1$, which completes the proof of Proposition \ref{p4.1}.
\end{proof}

\begin{rem}\label{r4.3}\rm In the proof of Proposition \ref{p4.1},
if $r<d(x,\,y)\le r+s$, then by \cite[p.\,42,\ Exercise 2.4.13]{bbi},
we also have $d(y, \ol{B(x,\,r)}\le d(y,B(x,\,r))
=d(x,\,y)-r\le s$, which is another proof of this fact.
\end{rem}

Now we give an equivalent characterization for the
$\dz$-annular decay property. First, we introduce the following
notion.

\begin{defn}\label{d4.2}\rm
Let $\tau\in[1,\,\fz)$. A doubling metric measure space
$(\cx,\,d,\,\mu)$ is said to have {\it Property $(P)_\tau$}, if
there exist positive constants $\dz$ and $C_{(P)_\tau}$ such that
for all $x\in\cx$, $s\in (0,\fz)$ and $r\in (\tau s, \fz)$,
\begin{equation}\label{e4.4}
\mu(B(x,\,r+s))-\mu(B(x,\,r))\le C_{(P)_\tau}\lf(\frac
sr\r)^{\dz}\mu(B(x,\,r)).
\end{equation}
\end{defn}

\begin{rem}\label{r4.4}\rm
(i) When $\tau=1$, it was proved in \cite[Remark 1.1]{mny10} that if
$\cx$ contains no less than two elements, then $\dz\in(0,\,1]$.
Hence, if $\cx$ contains no less than two elements, Property $(P)_1$
is just the $\dz$-annular decay property and we denote it simply by
Property $(P)$. Also, we denote the {\it corresponding constant
$C_{(P)_1}$} in \eqref{e4.4} by $C_P$.

(ii) Observe that if $r\in (0,\tau s]$, then \eqref{e4.4} is a
simple conclusion of the doubling property \eqref{e2.1} of $\mu$.
Moreover, if $r\in (0,\,s]$, then \eqref{e4.4} is always true, which
explains why we restrict that $\tau\in [1,\fz)$ in Definition
\ref{d4.2}.
\end{rem}

It is easy to show that Property $(P)_{\tau}$ with $\tau\in (1,\fz)$ is
equivalent to the $\dz$-annular decay property in the meaning as in
the following Proposition \ref{p4.2}. We omit the details.
In what follows, for any $a\in\rr$, we denote by $\lceil a\rceil$
the {\it smallest integer no less than $a$}.

\begin{prop}\label{p4.2} Let $\tau\in (1,\fz)$. Then

(i) Property $(P)$ implies Property $(P)_{\tau}$ with
$C_{(P)_\tau}\equiv C_P$.

(ii) Property $(P)_{\tau}$ implies Property $(P)$ with
$C_P\equiv(\lcz\tau\rcz)^{1-\dz}C_{(P)_\tau} C_1$, where $C_1$ is
the same as in Definition \ref{d2.1}.
\end{prop}

Finally, we give several examples of doubling
metric measure spaces having the $\dz$-annular decay property.

\begin{example}\rm\label{example4.1}

(i) $(\rd,\,|\cdot|,\,dx)$, the $d$-dimensional Euclidean space
endowed with the Euclidean norm $|\cdot|$ and the Lebesgue measure
$dx$. It is easy to show that $(\rd,\,|\cdot|,\,dx)$ has the
$\dz$-annular decay property for all $\dz\in(0,\,1]$.

(ii) $(\rd,\,|\cdot|,\,w(x)dx)$, the $d$-dimensional Euclidean space
endowed with the Euclidean norm $|\cdot|$ and the measure $w(x)dx$,
where $w$ is an $A_\fz(\rd)$ weight (see \cite[p.\,401]{gf85}
for its definition) and $dx$ is the Lebesgue
measure. Let $w$ be an $A_\fz(\rd)$ weight and for any
Lebesgue measurable set $E$, let $w(E)\equiv\int_Ew(x)\,dx$. Then there
exist positive constants $C$ and $\dz\in(0,\,1]$ such that for all balls $B$
and measurable subsets $E$ of $B$,
$$\dfrac{w(E)}{w(B)}\le C\lf(\frac{|E|}{|B|}\r)^\dz$$
(see \cite[p.\,401, Theorem 2.9]{gf85}). Clearly this
inequality implies that $(\rd,\,|\cdot|,\,w(x)dx)$
has the $\dz$-annular decay property.

(iii) Mac\'ias, Segovia and Torrea \cite{mst} introduced the
condition $(H_\az)$ with $\az\in(0,\,1]$ on a space of homogeneous
type. Recall that a doubling metric measure space $(\cx,\,d,\,\mu)$
is said to satisfy {\it Condition $(H_\az)$} with $\az\in(0,\,1]$,
if there exists a positive constant $C$ such that for all $x\in\cx$,
$r\in (0,\fz)$ and $s\in (0,r)$,
$$\mu(B(x,\,r+s))-\mu(B(x,\,r-s))\le
C[\mu(B(x,\,r))]^{1-\az}[\mu(B(x,\,s))]^\az.$$ If $\cx$ is an
RD-space, namely, there exist constants $0<\kz\le n$ and $C\ge1$
such that for all $x\in\cx$ and $0<r<2\diam(\cx)$ and
$1\le\lz<2\diam(\cx)/r$,
$$C^{-1}\lz^\kz\mu(B(x,\,r))\le\mu(B(x,\,\lz r))\le
C\lz^n\mu(B(x,\,r))$$ (see \cite{yz09}), where
$\diam(\cx)\equiv\sup_{x,\,y\in\cx}d(x,\,y)$, then there exists a
positive constant $C$ such that for all $x\in\cx$, $s\in (0,\fz)$
and $r\in (s, \fz)$,
$$\mu(B(x,\,r+s))-\mu(B(x,\,r))\le C\lf(\frac{s}{r}\r)^{\kz\az}\mu(B(x,\,r)).$$
This shows that for an RD-space, Condition $(H_\az)$  with
$\az\in(0,\,1]$ implies the $\dz$-annular decay property.

(iv) $(\hh^n,\,d,\,dx)$, the $(2n+1)$-dimensional Heisenberg group
$\hh^n$ with a left-invariant metric $d$ and the Lebesgue measure
$dx$. Buckley \cite{b99} showed that $(\hh^n,\,d,\,dx)$ is
a doubling metric measure space having the $\dz$-annular decay
property for all $\dz\in(0,\,1]$.

(v) $({\bbg},\,d,\,\mu)$, the nilpotent Lie group ${\bbg}$ with a
Carnot-Carath\'eodory (control) distance $d$ and a left invariant
Haar measure $\mu$. Fix a left invariant Haar measure $\mu$ on
$\bbg$. Then for all $x\in \bbg$, $V_r(x)=V_r(e)$; moreover, there
exist $\kz$, $D\in(0, \fz)$ with $\kz\le D$ such that for all
$x\in\bbg$, $C^{-1} r^\kz\le V_r(x)\le Cr^\kz$ when $r\in(0, 1]$,
and $C^{-1} r^D\le V_r(x)\le Cr^D$ when $r\in(1, \fz)$; see
\cite{nsw85} and \cite{vsc92}. By Proposition \ref{p4.1}
and the fact that $({\bbg},\,d)$ is a length space, we
know that $({\mathbb \bbg},\,d,\,\mu)$ is a doubling metric measure
space having the $\dz$-annular decay property for some
$\dz\in(0,\,1]$.
\end{example}

{\bf Acknowledgements.} Dachun Yang would like to thank Professor
Romain Tessera for some helpful discussions on Section 4 of this
paper. All authors sincerely wish to express
their deeply thanks to the referee for her/his very
carefully reading and also her/his so many valuable
and suggestive remarks, especially some remarks on
the proof of Proposition \ref{p4.1}, which lead us to obtain Theorem \ref{t4.1}
and much improve the presentation of this
article.

\bigskip

\noindent Haibo Lin

\medskip

\noindent College of Science, China Agricultural University, Beijing
100083, People's Republic of China

\medskip

\noindent{\it E-mail address}: \texttt{haibolincau@126.com}

\bigskip

\noindent Eiichi Nakai

\medskip

\noindent Department of Mathematics, Osaka Kyoiku University,
Kashiwara, Osaka 582-8582, Japan

\medskip

\noindent{\it E-mail address}: \texttt{enakai@cc.osaka-kyoiku.ac.jp}

\bigskip

\noindent Dachun Yang (Corresponding author)

\medskip

\noindent School of Mathematical Sciences, Beijing Normal
University, Laboratory of Mathematics and Complex systems, Ministry
of Education, Beijing 100875, People's Republic of China

\medskip

\noindent{\it E-mail address}: \texttt{dcyang@bnu.edu.cn}

\end{document}